\documentclass{article}
\pdfoutput = 1

\usepackage[T1]{fontenc}        
\usepackage[utf8]{inputenc}
\usepackage[english]{babel}     

\usepackage[defaultlines=4,all]{nowidow}
\usepackage{csquotes}          
\usepackage{authblk} 
\usepackage{scalerel,stackengine}

\usepackage{amsmath,xparse}
\usepackage{amsfonts}
\usepackage{amssymb}
\usepackage{amsthm}

\usepackage[usenames,dvipsnames,table]{xcolor}  
\usepackage{mathrsfs}
\usepackage{graphicx}
\usepackage{stmaryrd}

\usepackage{lmodern}   
\usepackage{bm}               
\usepackage{bbold}

\usepackage[inline]{enumitem}
\setlist{itemsep=3pt,topsep=3pt,partopsep=2pt,parsep=2pt,leftmargin=30pt}
\setlist[2]{itemsep=2pt,topsep=0pt,partopsep=2pt,leftmargin=20pt,parsep=2pt}
\setlist[enumerate]{label*=\textup{\arabic*.}}

\SetEnumitemKey{ind}{label=\textup{(\roman*)}}
\SetEnumitemKey{dep}{label=\textup{(\alph*)}}
\SetEnumitemKey{seq}{label=\textup{(\arabic*)}}
\SetEnumitemKey{cases}{label=\textup{(\alph*)}}
\SetEnumitemKey{cond}{label=\textup{(\alph*)}}

\usepackage[
            unicode,
            colorlinks=true,
            linkcolor=blue,
            citecolor=cyan 
            ]{hyperref}

\usepackage[capitalise]{cleveref}     
\usepackage[capitalise]{cleveref}

 \crefname{def}{definition}{definitions}
 \Crefname{def}{Definition}{Definitions}

 \crefname{thm}{theorem}{theorems}
 \Crefname{thm}{Theorem}{Theorems}

 \crefname{prop}{proposition}{propositions}
 \Crefname{prop}{Proposition}{Propositions}

 \crefname{lem}{lemma}{lemmas}
 \Crefname{lem}{Lemma}{Lemmas}

 \crefname{exm}{example}{examples}
 \Crefname{exm}{Example}{Examples}

 \crefname{cor}{corollary}{corollaries}
 \Crefname{cor}{Corollary}{Corollaries}

 \crefname{eq}{equation}{equations}
 \Crefname{eq}{Equation}{equations}

 \crefname{defn}{definition}{definitions}
 \Crefname{defn}{Definition}{Definitions}

 \crefname{qst}{question}{questions}
 \Crefname{qst}{Question}{Questions}

 \crefname{eq}{equation}{equations}
 \Crefname{eq}{Equation}{Equations}

 \crefname{rem}{remark}{remarks}
 \Crefname{rem}{Remark}{Remarks}

 \crefname{nott}{notation}{notations}
 \Crefname{nott}{Notation}{Notations}

 \crefname{cnj}{conjecture}{conjectures}
 \Crefname{cnj}{Conjecture}{Conjectures}
 
 \crefname{fig}{figure}{figures}
 \Crefname{fig}{Figure}{Figures}
 
 \crefname{claim}{claim}{claims}
 \Crefname{claim}{Claim}{Claims}

\usepackage{float}
\usepackage{tikz-cd}
\tikzset{close/.style={near start,outer sep=-2pt}}

\usepackage{tikz}
\usetikzlibrary{automata,positioning,decorations.markings}
\usetikzlibrary{decorations.pathmorphing}
\usetikzlibrary{decorations.pathreplacing,calligraphy}
\usetikzlibrary{arrows}
\usetikzlibrary{shapes.geometric}
\tikzset{dot/.style={circle,fill=black,inner sep=0pt, minimum size=0.8pt}}
\tikzset{smallstate/.style={circle,fill=black,inner sep=0pt, minimum size=3pt}}

\usepackage{tikz-cd}

\usepackage{xifthen}
\usepackage{scalerel}

\usepackage[a4paper,left=3cm,top=3.7cm,right=3cm,bottom=3.1cm]{geometry}   
\usepackage[backend = biber,       
            language = english,
            style    = numeric-comp,   
            giveninits = true,
            isbn = false,
            doi = false,
            sorting = nyt,
            backref=true,
            sortcites
            ]{biblatex}

\DeclareNameAlias{sortname}{last-first}
\renewbibmacro{in:}{%
  \ifentrytype{article}{}{%
  \printtext{\bibstring{in}\intitlepunct}}}
\AtEveryBibitem{
    \clearfield{urlyear}
    \clearfield{urlmonth}
    \clearfield{note}%
    \clearfield{url}
}

\DeclareNameAlias{labelname}{given-family}
\renewrobustcmd*{\bibinitdelim}{\,}

\newcommand{\citenr}[1]{\textnormal{\citeauthor{#1},~\cite{#1}}}

\newtheorem{thm}{Theorem}[section]
\newtheorem{cor}[thm]{Corollary}
\newtheorem{lem}[thm]{Lemma}
\newtheorem{prop}[thm]{Proposition}

\theoremstyle{definition}
\newtheorem{defn}[thm]{Definition}
\newtheorem{exm}[thm]{Example}
\newtheorem{que}[thm]{Question}
\theoremstyle{remark}
\newtheorem{rem}[thm]{Remark}
\numberwithin{equation}{section}

\newcommand{\ie}{i.e.,~}       
\newcommand{\eg}{e.g.~}        

\newcommand{\wrt}{w.r.t.~}
\newcommand{\resp}{resp.,~}

\newcommand{\rk}{\mathrm{rk}}
\newcommand{\im}{\mathrm{Im}}

\newcommand{\ZZ}{\mathbb{Z}}
\newcommand{\NN}{\mathbb{N}}
\newcommand{\trivial}{1}                
\newcommand{\Trivial}{\{ \trivial \}}   
\newcommand{\normalcl}[1]{\langle \! \langle \hspace{1pt} #1 \hspace{1pt} \rangle \! \rangle}

\newcommand{\fgen}{\ensuremath{\mathsf{fg}}}
\newcommand{\fin}{\ensuremath{\mathsf{fi}}}

\newcommand{\leqfg}{\leqslant_{\fgen}}
\newcommand{\leqfi}{\leqslant_{\fin}}

\newcommand{\FTA}{\Fn \times \ZZ^m}

\newcommand{\fg}{0}
\newcommand{\nfg}{1}

\newcommand{\parts}{\mathcal{P}([k])\setminus \{\varnothing\}}

\newcommand{\schi}{\raisebox{1pt}{$\chi$}}
\newcommand{\one}{\mathsf{1}}
\newcommand{\zero}{\mathsf{0}}
\newcommand{\conf}[1][]{{\schi}_{\scriptstyle{#1}}}

\newcommand{\qplus}{\mathbin{\scalebox{0.9}{$\boxplus$}}}

\newcommand{\join}{\mathop{
\mathchoice
  {\textstyle{\vee}}
  {\textstyle{\scalebox{0.8}{$\vee$}}}
  {\scriptscriptstyle{\vee}}
  {\scriptscriptstyle{\vee}}
}}

\newcommand{\isom}{\simeq}      

\newcommand{\Free}[1][]{%
\ifthenelse{\isempty{#1}}{\mathbb{F}}{\mathbb{F}_{\!#1}}
}
\newcommand{\Fn}{\Free[n]}
\newcommand{\Zm}{\mathbb{Z}^m}

\newcommand{\cone}[2]{c_{#2}(#1)}

\newcommand{\defin}[1]{\emph{#1}\index{#1}}
\renewcommand{\t}{\mathrm{t}}
\newcommand{\Basis}{\mathcal{B}}

\newcommand{\preim}{
^{\tikz[baseline=-0.4ex]{\draw [{<[scale=.75]}-] (0,0) -- (.17,0);}}
  }

\newcommand{\ssum}{\scalebox{0.7}{$\sum$}}

\newcommand{\gen}[1]{\langle  #1 \rangle}

\newcommand{\normaleq}{ \trianglelefteqslant }
\newcommand{\st}{\mid}                       
\newcommand{\cab}[2]{\mathbf{c}_{#2}(#1)} 
\newcommand{\Cab}[2]{C_{#2}(#1)} 

\newcommand{\xto}[2][]{\xrightarrow[#1]{#2}}

\newcommand{\id}{\operatorname{id}}

\tikzset{every state/.style={
    inner sep=1pt,
    minimum size=4pt,
    fill=black
    }}

\newcommand{\bp}{\ensuremath{\begin{tikzpicture}   
\node[state,accepting] (1) {}; \end{tikzpicture}
}}

\newsavebox{\mytempbox}

\newcommand{\xinto}[1]{%
   \sbox{\mytempbox}{\hbox{\( \scriptstyle\mkern10mu#1\mkern12mu \)}}
   \tikz[baseline=-0.5ex]{\draw[right hook->,line cap=round] (0,0) --
     node[midway,above=-0.3ex]{\usebox\mytempbox} (\wd\mytempbox,0);}
  }
  
\newcommand{\xonto}[1]{%
    \sbox{\mytempbox}{\hbox{\( \scriptstyle\mkern10mu#1\mkern12mu \)}}
    \tikz[baseline=-0.5ex]{\draw[->>,line cap=round] (0,0) --
      node[midway,above=-0.3ex]{\usebox\mytempbox} (\wd\mytempbox,0);}
   }

\newcommand{\into}{\xinto{\ }}
\newcommand{\onto}{\xonto{\ }}

\newcommand{\cast}{\circledast}

\stackMath
\newcommand\reallywidehat[1]{%
\savestack{\tmpbox}{\stretchto{%
  \scaleto{%
    \scalerel*[\widthof{\ensuremath{#1}}]{\kern-.6pt\bigwedge\kern-.6pt}%
    {\rule[-\textheight/2]{1ex}{\textheight}}
  }{\textheight}%
}{0.5ex}}%
\stackon[1pt]{#1}{\tmpbox}%
}

\newcommand{\set}[1]{\{  #1  \}}        

\newcommand{\SIP}{\ensuremath{\mathsf{SIP}}}
\newcommand{\MSIP}{\ensuremath{\mathsf{MSIP}}}

\title{Intersection configurations in free and\\free times free-abelian groups}

\author[1]{J. Delgado\thanks{\url{jorge.delgado@upc.edu}}}

\author[2]{M. Roy\thanks{\url{mallikaroy75@gmail.com}}}

\author[3]{E. Ventura\thanks{\url{enric.ventura@upc.edu}}}
\affil[1,3]{Departament de Matem\`atiques, Universitat Polit\`ecnica de Catalunya, and Institut de Matem\`atiques de la UPC-BarcelonaTech, Catalonia}
\affil[2]{Departamento de Matemáticas, Universidad del País Vasco (UPV/EHU), Spain}

\begin{document}

\maketitle

\begin{abstract}
In this paper we study intersection configurations ---which describe the behaviour of multiple (finite) intersections of subgroups with respect to finite generability--- in the realm of free and free times free-abelian (FTFA) groups. We say that a configuration is realizable in a group $G$ if there exist subgroups $H_1,\ldots , H_k \leqslant G$ realizing it.

It is well known that free groups $\Fn$ satisfy the Howson property: the intersection of any two finitely generated subgroups is again finitely generated. We show that the Howson property is indeed the only obstruction for multiple intersection configurations to be realizable within nonabelian free groups.

On the contrary, FTFA groups $\Fn \times \Zm$ are well known to be non-Howson. We also study multiple intersections within FTFA groups, providing an algorithm to decide, given $k\geq 2$ finitely generated subgroups, whether their intersection is again finitely generated and, in the affirmative case, compute a `basis' for it. We finally prove that any intersection configuration
is realizable in a FTFA group $\FTA$, for $n\geq 2$ and big enough $m$. As a consequence, we exhibit finitely presented groups where every intersection configuration is realizable.
\end{abstract}

\medskip


\vspace{15pt}
\section{Introduction}

The behaviour of subgroup intersections with respect to finite generability has a long and interesting history in the context of nonabelian groups. Although finitely generated free groups $\Free[n]$ contain non-(finitely generated) subgroups, in the 1950's \citeauthor{howson_intersection_1954} proved that the intersection of two (and hence, finitely many) finitely generated subgroups of $\Free[n]$ is again finitely generated; see~\cite{howson_intersection_1954}. In acknowledgment of this initial result, a group is said to satisfy the \defin{Howson property} (or to be \defin{Howson}, for short) if the intersection of any two finitely generated subgroups is again finitely generated. As a generalization, B. Baumslag proved in~\cite{baumslag_intersections_1966} the preservation of the Howson property under free products: if $G_1$ and $G_2$ satisfy the Howson property then so does $G_1 *G_2$. However, the same statement fails dramatically when replacing the free product by the apparently tame direct product: the result below is folklore (it appears in~\cite{burns_intersection_1998} attributed to Moldavanski, and as the solution to Exercise~23.8(3) in~\cite{bogopolski_introduction_2008}).

\begin{prop}\label{prop: Fn x Z^n no Howson}
The group $\FTA$, for $n\geq 2$ and $m\geq 1$, does not satisfy the Howson property.
\end{prop}

\begin{proof}
In $\Free[2] \times \ZZ =\langle x,y \mid \, \rangle \times \langle t \mid \, \rangle $, consider the (two-generated, free) subgroups $H=\langle x,y \rangle$ and $H'=\langle xt, y \rangle$. Clearly,
 \begin{align*}
H\cap H' & =\{ w(x,y) \mid w\in \Free[2]\} \cap \{ w(xt,y) \mid w\in \Free[2]\} \\ & = \{ w(x,y) \mid w\in \Free[2]\} \cap \{ w(x,y)\t^{|w|_x} \mid w\in \Free[2]\} \\ & = \{ w(x,y)\t^0 \mid w\in \Free[2],\,\, |w|_x =0 \} \\ & =\langle x^{-k} y x^{k} ,\, k\in \ZZ \rangle =\langle\!\langle y\rangle\!\rangle_{\Free[2]},
 \end{align*}
where $|w|_x$ is the total $x$-exponent of $w$ (i.e.\ the first coordinate of the image $\mathbf{w} = (w)\rho\in \ZZ^2$ of $w \in \Free[2]$ under the abelianization map $\rho \colon \Free[2]\onto \ZZ^2$ with the obvious bases). It is well known that the normal closure $\normalcl{y}$ of $y$ in $\Free[2]$ is not finitely generated, hence $\Free[2] \times \ZZ$ does not satisfy the Howson property. Since, for all $n\geq 2$ and $m\geq 1$, the group $\Free[2] \times \ZZ$ embeds in $\Fn \times\ZZ^m$, this last one does not satisfy the Howson property either.
\end{proof}

We remark that the subgroups $H$ and $H'$ in the previous counterexample are both isomorphic to $\Free[2]$. Quite interestingly, the above is a situation where two free groups of rank 2 have a non-(finitely generated), of course free, intersection. This does not contradict the Howson property for free groups, but rather indicates that there is no free subgroup of $\Free[2] \times \ZZ$ containing both $H$ and~$H'$.

The behavior of intersections within free times free-abelian groups (FTFA groups, for short) was studied in detail in~\cite{delgado_algorithmic_2013,delgado_stallings_2022}, where the authors solved the so-called 
\defin{Subgroup Intersection Problem} (\SIP, for short) within this family of groups. 

\begin{thm}[\citenr{delgado_algorithmic_2013}]\label{thm: SIP FTA}
The Subgroup Intersection Problem for $\FTA$ is computable. \qed
\end{thm}

The $\SIP$ is the special case of the Multiple Subgroup Intersection Problem for $k=2$ subgroups:

\begin{defn}
The \defin{Multiple Subgroup Intersection Problem for a group $G$}, $\MSIP(G)$, consists in, given finite sets of generators for finitely many subgroups $H_1,\ldots ,H_k\leqfg G$, deciding whether the intersection ${H_1\cap \cdots \cap H_k}$ is again finitely generated and, in the affirmative case, computing a set of generators for it.  
\end{defn}

In the present paper we investigate the multiple versions of both the Howson property for free groups, and the Subgroup Intersection Problem for FTFA groups (solved in the theorem above for two subgroups). We emphasize that $\MSIP(\FTA)$ does not follow directly by induction from $\SIP(\FTA)$ (using the recurrence $H_1\cap \cdots \cap H_k =(H_1\cap \cdots \cap H_{k-1}) \cap H_k$) because it could very well happen that some of the intermediate intersections, even all of them, are not finitely generated, whereas $H_1, \ldots ,H_k$ and $H_1\cap \cdots \cap H_k$ are all finitely generated. Instead, we need to build a procedure dealing directly with the total intersection $H_1\cap \cdots \cap H_k$, but without going through the intermediate ones.

At the beginning of \Cref{sec: multiple intersection} we briefly survey the algebraic proof for \Cref{thm: SIP FTA}, before delving into the more involved multiple variant of this problem. Then, we set up the machinery needed to study multiple intersections within $\FTA$, which allows us to extend $\SIP(\FTA)$ to $\MSIP(\FTA)$ (see \Cref{thm: MSIP}) and prove some technical statements (remarkably, \Cref{thm: technical}) crucial to derive our main results in later sections.

On the other hand recall that, since $\FTA$ is not Howson (for $n \geq 2$ and $m\geq 1$), for each pair of subgroups of $\FTA$, there are two possibilities for their intersection: either it is finitely generated, or it is not. When we consider $k\geq 2$ subgroups, many different combinations of finitely generated and non-(finitely generated) partial intersections may arise. In \Cref{sec: conf} we introduce the notion of intersection configuration as a compact way to describe all the possible intersection situations (in terms of finite generability) between finitely many subgroups.

In Section~\ref{sec: unrealizable} we take advantage of our analysis of multiple intersections to deduce obstructions for these $k$-intersection configurations to be realizable in $\FTA$. In particular, we see that, despite their very flexible character (described in \cite{delgado_stallings_2022}) not every intersection configuration is realizable in (a fixed) $\FTA$.

As a natural continuation, in Section~\ref{sec: realizing} we use some results from \Cref{sec: multiple intersection} to show that every $k$-configuration is realizable in $\FTA$ for a big enough $m$. That is, for $\mathcal{I} \subseteq \mathcal{P}([k])\setminus \{ \varnothing\}$, there exists a big enough $m\geq 0$ and subgroups $H_1, H_2, \ldots ,H_k$ of  $\FTA$ satisfying the following: for every nonempty $I\subseteq \{1,\ldots ,k\}$, the intersection $H_I=\bigcap_{\, i\in I} H_i$ is finitely generated if and only if $I \in \mathcal{I}$ (see Theorem~\ref{thm: all realizable} below for details). We deduce the existence of finitely presented groups where all such configurations are realizable; we call them intersection-saturated groups. 

Finally, in Section~\ref{sec: free}, we study the free case ($m=0$): such a $k$-configuration is realizable in $\Fn$, $n\geq 2$, if and only if, for every nonempty $I,J\subseteq \{1,\ldots ,k\}$, $H_{I\cup J}=H_I\cap H_J$ is required to be finitely generated whenever $H_I$ and $H_J$ are so; that is, the Howson property is the \emph{only} obstacle to realize arbitrary $k$-configurations in a free ambient group; see~\Cref{thm: char Fn}.

\subsection*{General notation and conventions}

The set of natural numbers, denoted by $\NN$, is assumed to contain 0, and we specify conditions on this set using subscripts; for example, we denote by $\NN_{\geq 1}$ the set of strictly positive integers. For $k \in \NN_{\geq 1}$, we write $[k]=\{n\in \NN \st 1 \leq n \leq k\}$.

We use lowercase boldface font ($\mathbf{a}, \mathbf{b}, \mathbf{c},\ldots$) to denote elements of the free-abelian group $\ZZ^m$ (usually thought as horizontal vectors); and uppercase boldface font ($\mathbf{A},\mathbf{B},\mathbf{C},\ldots $) to denote matrices, which --- as homomorphisms in general --- are assumed to act on the right. That is, we denote by $(x)\varphi$ (or simply by~$x \varphi$) the image of the element $x$ by the map $\varphi$, and we denote by $\varphi \psi$ the composition \smash{$A \xto{\varphi} B \xto{\psi} C$}. In order to distinguish from inverse maps, we denote by $S\varphi\preim$ the set of preimages in $A$ of the set $S\subseteq B$ under the map $\varphi$.

Throughout the paper we write $H \leqslant G$ (\resp $H\leqfg G$ and $H\leqfi G$) to express that $H$ is a subgroup (\resp a finitely generated subgroup, and a finite index subgroup) of $G$, reserving the symbol $\leq$ for inequalities among real numbers.
We use the abbreviations `f.g.' to mean `finitely generated' and `non-f.g.' to mean `non-(finitely generated)'; and we refer to the difference between finite generability and non-(finite generabilty) of a subgroup as the \defin{character} of the subgroup. 

Finally, the symbol $\infty$ denotes the \emph{countable} infinity (\ie $\infty =\aleph_0$), and we denote by $\ZZ^{\infty}$ the \emph{direct sum} of countably many copies of $\ZZ$, \ie $\ZZ^{\infty}=\bigoplus_{n\geq 1} \ZZ$ (where elements always have finite support).
More specific notation and terminology are set forth in the corresponding sections.

\section{Multiple intersections in free times free-abelian groups}\label{sec: multiple intersection}

We call \defin{free times free-abelian} (FTFA) groups the groups admitting a presentation of the form
 \begin{equation} \label{eq:pres FTFA}
\FTA =\left\langle \, x_1,\ldots,x_n,t_1,\ldots,t_m\, \mid t_it_j \,=\, t_jt_i,\, t_ix_k=x_kt_i \, \right\rangle,
 \end{equation}
where $n,m \geq 0$.
We write $X=\{x_1,\ldots,x_n\}$ the set of (freely independent) generators of a free group $\Fn$ of rank $n$, and $T=\{t_1,\ldots,t_m\}$ the set of (commuting and linearly independent) generators of a free-abelian group $\Zm$ of rank $m$; of course, each of the $x_i$'s also commutes with each of the $t_j$'s.
Note that these groups $G = \Fn \times \Zm$ fit in the middle of a short exact sequence of the form
 \[
\begin{tikzcd} 1 \arrow[r] & \Zm \arrow[r,"\iota"] & G \arrow[r,"\pi"] & \Fn \arrow[r] \arrow[l, dashed, bend left,"\sigma"] & 1, \end{tikzcd}
 \]
which obviously splits; that is, there exists a homomorphism $\sigma \colon \Fn \to G$ such that $\sigma \pi =\id_{\Fn}$ (called a \defin{section} of $\pi$).

As is customary with this kind of groups, we shall refer to the elements in~$G = \FTA$ by using their normal forms (with vectors on the right), which we write multiplicatively as $u\t^\mathbf{a}$, where $u=u(x_1,\ldots ,x_n) \in \Fn$ is called the \emph{free part} of $u\t^\mathbf{a}$, and $\mathbf{a}=(a_1, \ldots,a_m)\in \ZZ^m$ is called the \emph{abelian part} of $u\t^\mathbf{a}$ (the meta-symbol $\t$ is just a mnemonic way to encapsulate the standard additive notation for $\ZZ^m$ into a multiplicative one, i.e.\ $\t^\mathbf{a}=t_1^{a_1} \cdots t_m^{a_m}$). Note that then, $(u\t^\mathbf{a})(v\t^\mathbf{b})=uv \, \t^\mathbf{a+b}$.

We denote by $\pi\colon \FTA \onto \Fn$, $u\t^\textbf{a}\mapsto u$ and by $\tau\colon \FTA \onto \ZZ^m$, $u\t^\textbf{a}\mapsto \textbf{a}$, the natural projections to the free part $\Fn$ and to the free-abelian part $\ZZ^m$ (now in additive notation), respectively. Clearly, both maps are group homomorphisms.

It is not difficult to see that every subgroup $H \leqslant \Fn \times \Zm$ is of the form $H=H\pi \sigma \times (H \cap \Zm)$, where $\sigma$ is a section of $\pi$, $\rk (H\pi\sigma) =\rk (H \pi)\in [0,\infty]$, and the rank of $L_H =H\cap \Zm$ is at most $m$. An immediate (but important for us) consequence is stated below.

\begin{rem}\label{rem: H fg iff Hpi fg}
A subgroup $H\leqslant \Fn \times \ZZ^m$ is finitely generated if and only if its projection $H \pi$ (to the free part) is finitely generated.
\end{rem}

Next, we introduce (abelian) completions, a concept playing an important role in this kind of groups.

\begin{defn} \label{def: completion}
Let $H$ be a subgroup of $G =\FTA$, let $\sigma$ be a section of $\pi_{|H}$, let $w\in \Free_n$, and let $\mathbf{a} \in \Zm$. If $w \t^{\mathbf{a}} \in H$ then we say that $\mathbf{a}$ \defin{completes $w$ into} $H$. More precisely, the \defin{$\sigma$-completion} of $w$ in $H$, denoted by $\cab{w}{H,\sigma}$, is the empty set if $w\not\in H\pi$, and equals to $w\sigma \tau \in \ZZ^m$ otherwise. Similarly, the \defin {(full) completion} of $w$ in $H$ is $\Cab{w}{H}=(w\pi_{|H}\preim) \tau =\{ \mathbf{a} \in \ZZ^m \st w\t^{\mathbf{a}} \in H \}$. That is, $\Cab{w}{H}$ is the full set of vectors completing $w$ into~$H$.
\end{defn}

It is straightforward to see that $\Cab{w}{H}$ is either empty or a coset of $L_{H}=H\cap \Zm$.

\begin{lem}\label{lem: completions are cosets}
Let $H$ be a subgroup of $\FTA$, and let $w\in \Free$. Then,
 \begin{equation}
\Cab{w}{H} \,=\, \bigg\{\! \begin{array}{ll} \varnothing & \text{if } w\notin H \pi, \\ \cab{w}{H,\sigma} + L_{H} & \text{if } w \in H \pi \,, \end{array} 
 \end{equation}
where $\sigma$ is any section of $\pi_{|H}$. \qed
\end{lem}

As defined in~\cite{delgado_algorithmic_2013}, a \emph{basis} of a subgroup $H\leqslant \FTA$ is a subset of $H$ of the form
 \begin{equation} \label{eq: FTFA basis}
\Basis \,=\, \left\{ u_1\t^\mathbf{a_1}, u_2\t^\mathbf{a_2},\ldots ; \t^\mathbf{b_1},\ldots, \t^\mathbf{b_s} \right\}\,,
 \end{equation}
where $\{u_1, u_2, \ldots \}$ is a (finite or infinite) \emph{free basis} of $H\pi\leqslant \Free[n]$, $\mathbf{a_1}, \mathbf{a_2},\ldots \in \ZZ^m$ are integral vectors, and $\{\t^\mathbf{b_1}, \ldots,\t^\mathbf{b_s}\}$ is a \emph{free-abelian basis} of the intersection $L_H =H\cap \ZZ^m$; we use a semicolon as a notational device to separate the purely abelian elements in $\Basis$. In~\cite[Sect. 1]{delgado_algorithmic_2013} it is shown that every subgroup $H\leqslant \FTA$ admits such a basis (computable from any given set of generators in the finitely generated case). We use the terms \emph{free basis}, \emph{free-abelian basis} and just \emph{basis}, depending on whether we refer to a basis for the free part, the free-abelian part, or the whole FTFA group, respectively.

We denote $r=\rk(H\pi)$ and $s=\rk(L_H)$. Note that $0\leq r\leq \infty$, while $0\leq s\leq m$, with the first dots in \Cref{eq: FTFA basis} representing countably many elements in the case $r=\infty$. Note that, then, $H\isom \Free[r] \times \ZZ^s$, and $H$ is finitely generated if and only if $r<\infty$. Finally, we introduce some more notation for later use: let $\rho_H\colon H\pi \onto \ZZ^{r}$ be the abelianization map (not to be confused with the restriction to $H$ of the global abelianization map $\Free[n] \onto \ZZ^n$), and let $A_H\colon \ZZ^r\to \ZZ^m$ be the so-called \defin{completion homomorphism}, sending the canonical $j$-th vector to $\mathbf{a_j}$. 

When $r=\rk (H\pi)<\infty$, we write
 \begin{equation*}
\mathbf{A}_H=\left( \!\begin{array}{c} \mathbf{a_1} \\ \vdots \\ \mathbf{a_r} \end{array}\!\right) \in M_{r\times m}(\ZZ), \quad  \quad \mathbf{L}_H =\left( \!\begin{array}{c} \mathbf{b_1} \\ \vdots \\ \mathbf{b_s} \end{array}\!\right) \in M_{s\times m}(\ZZ),
 \end{equation*}
called the \defin{completion matrix} of $H$ (\wrt $\Basis$), and the \defin{matrix} of $L_H$ (\wrt $\Basis$), respectively; note that $\mathbf{A}_H$ is the matrix of the completion homomorphism $A_H$ \wrt the canonical bases, and that the row space of $\mathbf{L}_H$ (\ie the subspace of $\ZZ^m$ generated by the rows of $\mathbf{L}_H$) is precisely $L_H$.

Below we see that full completions are easily computable from any given finite basis for the subgroup.

\begin{cor} \label{cor: compl computable}
If $\Basis =\big\{ u_1\t^\mathbf{a_1}, u_2 \t^{\mathbf{a_2}},\ldots ;\t^\mathbf{b_1},\ldots, \t^\mathbf{b_s} \big\}$ is a basis of $H\leqslant \FTA$ then, for every $w\in H\pi$, $\Cab{w}{H}=w\rho_H A_H +\gen{\mathbf{b_1}, \ldots,\mathbf{b_s}}$. 
\end{cor}

\begin{proof}
Let $\sigma$ be the section of $\pi_{\mid H}$ given by $\Basis$; namely, $\sigma \colon H\pi \to H$, $u_i \mapsto u_i \t^{\mathbf{a_i}}$. Then, applying \Cref{lem: completions are cosets} to $\sigma$ we have that, for every $w\in H\pi$, $\Cab{w}{H}=w\sigma \tau +L_{H}=w\rho_H A_H +\gen{\mathbf{b_1}, \ldots, \mathbf{b_s}}$, since $\sigma\tau=\rho_H A_H \colon H\pi \to \ZZ^m$, $u_i\mapsto \mathbf{a_i}$.
\end{proof}

In~\cite{delgado_algorithmic_2013}, Delgado and Ventura studied intersections between \emph{two} finitely generated subgroups of $\FTA$ given by respective bases (alternatively, see~\cite{delgado_stallings_2022,delgado_extensions_2017} for a geometric description of the subgroups of $\FTA$ and their intersections in the spirit of Stallings automata). We briefly summarize the main results here, with the goal of generalizing them to the case of \emph{finitely many}, say $k\geq 2$, \emph{arbitrary} subgroups $H_1,\ldots ,H_k\leqslant \FTA$.

Consider two subgroups $H_1,H_2\leqslant \FTA$ and take (finite or infinite) bases for them 
 \begin{align*}
H_1 & \,=\, \langle u_1 \t^\mathbf{a_1}, u_2 \t^\mathbf{a_2},\ldots ; \t^\mathbf{b_1},\ldots,\t^\mathbf{b_{s_1}}\rangle, \\ H_2 &\,=\, \langle v_1 \t^\mathbf{a'_1}, v_2 \t^\mathbf{a'_2},\ldots ; \t^\mathbf{b'_1},\ldots,\t^\mathbf{b'_{s_2}} \rangle,
 \end{align*}
where $\{u_1, u_2, \ldots \}$ and $\{v_1,v_2,\ldots \}$ are (finite or infinite) sets of freely independent elements in~$\Fn$, $\mathbf{a_i}, \mathbf{a'_i}, \mathbf{b_i}, \mathbf{b'_i}\in \ZZ^m$, and $\{\t^\mathbf{b_1},\ldots,\t^\mathbf{b_{s_1}}\}$ and $\{ \t^\mathbf{b'_1}, \ldots, \t^\mathbf{b'_{s_2}}\}$ are free-abelian bases for $L_1=H_1\cap \ZZ$ and $L_2=H_2\cap \ZZ$, respectively. Observe that $(H_1\cap H_2)\pi\leqslant H_1\pi \cap H_2\pi$, with the equality not being true in general: $H_1\pi \cap H_2\pi$ is the set of elements from $\Free[n]$ admitting completions both in $H_1$ and in $H_2$, while $(H_1\cap H_2)\pi$ is the set of elements from $\Free[n]$ which admit a \emph{common} completion in both $H_1$ and $H_2$. In~\cite{delgado_algorithmic_2013}, the diagram in \Cref{fig: int 2} was used in order to describe the key subgroup $(H_1\cap H_2)\pi$.

\begin{figure}[H]
\centering
\begin{tikzcd}[row sep=25pt, column sep=25pt,ampersand replacement=\&]
\& \& (H_1 \cap H_2) \pi \\[-28pt]
\& \& \rotatebox[origin=c]{270}{$\normaleq$} \\[-28pt]
\& H_1 \pi \arrow[d,->>,"\rho_1"]
\&H_1 \pi \cap H_2 \pi \arrow[d,->>,"\rho"]
 \arrow[r,hook,"\iota_2"] \arrow[l,hook',"\iota_1"']
    \& H_2\pi \arrow[d,->>,"\rho_2"]
    \\[10pt]
\&\ZZ^{r_1} \arrow[rd,->,"A_1"']   \&\ZZ^{r}
    \arrow[ur, phantom, "///"]
    \arrow[ul, phantom, "///"]
    \arrow[d,->,"R"]
    \arrow[r,->,"P_2"]
    \arrow[l,->,"P_1"']\&
    \ZZ^{r_2} \arrow[ld,->,"A_2"]\\
\& \&\Zm\\[-28pt]
\& \& \rotatebox[origin=c]{90}{$\leqslant$} \\[-28pt]
\& \& L_1 + L_2
   \end{tikzcd}
   \caption{Intersection diagram for $H_1 \cap H_2$} \label{fig: int 2}
\end{figure}

Here, $\iota_1$ and $\iota_2$ are the natural inclusions; $\rho_1$, $\rho_2$ and $\rho$ are the corresponding abelianization maps with respect to the chosen bases in the domains and codomains (not to be confused with the restrictions of the global abelianization $\Free[n] \onto \mathbb{Z}^n$); $P_1$ and $P_2$ are the abelianizations of $\iota_1$ and $\iota_2$ (note that, although $\iota_1$ and $\iota_2$ are injective, $P_1$ and $P_2$ may very well not be so); and $A_1$ and $A_2$ are the completion homomorphisms, sending the $i$-th canonical vector to $\mathbf{a_i}$ and to $\mathbf{a'_i}$, respectively. In general, the ranks $r_1=\rk(H_1\pi)$, $r_2=\rk(H_2\pi)$, and $r=\rk(H_1\pi\cap H_2\pi)$ may be finite or infinite, with the restriction coming from the Howson property for free groups: $r$ is finite whenever $r_1$ and $r_2$ are so. Finally, $R$ is the linear map given by $R=P_1A_1-P_2A_2$. Using this scheme, Delgado--Ventura~\cite{delgado_algorithmic_2013} proved the following result, announced there for finitely generated subgroups $H_1$, $H_2$, but valid in full generality with essentially the same proof; see also \cite[Section 4]{delgado_stallings_2022} for a more detailed (geometric) analysis of this and related facts.

\begin{prop}[{Delgado--Ventura, \cite{delgado_algorithmic_2013}}]\label{prop: 2-int}
For any two subgroups $H_1,H_2\leqslant \FTA$ (and using the above notation), $(H_1 \cap H_2) \pi =(L_1 +L_2)R\preim \rho\preim$, which is a normal subgroup of $H_1 \pi \cap H_2\pi$. $\Box$
\end{prop}
It is well-known that, for $r<\infty$, a normal subgroup $N\normaleq \Free[r]$ is finitely generated if and only if either $N=\Trivial$ or $N\leqfi \Free[r]$; whereas, a normal subgroup $N\normaleq \Free[\infty]$ is finitely generated if and only if $N= \Trivial$. From this observation and \Cref{prop: 2-int}, we can immediately deduce the following characterization for the finite generability of the intersection of two arbitrary subgroups of $\FTA$, fully in abelian terms; see~\cite[Lemma 4.17]{delgado_stallings_2022} for details.

\begin{cor}[{Delgado--Ventura, \cite{delgado_algorithmic_2013}}] \label{cor: 2-int fg}
For any two subgroups $H_1,H_2\leqslant \FTA$ (and using the above notation), $H_1 \cap H_2$ is finitely generated if and only if either (i) $r=0,1$; 
or (ii) $2\leq r<\infty$ and $(L_1 +L_2)R\preim \leqfi \ZZ^r$.\qed
\end{cor}

\begin{rem}
Note that if $r=\rk(H_1\pi \cap H_2\pi) = \infty$ then $(H_1 \cap H_2) \pi$ is never finitely generated, since $(H_1 \cap H_2) \pi = (L_1 +L_2)R\preim \rho\preim \geqslant [\Free[\infty], \Free[\infty]] \neq \Trivial$ and hence is a nontrivial normal subgroup of $\Free[\infty]$, thus not being finitely generated.
\end{rem}

When $H_1$ and $H_2$ are finitely generated, each of $r_1,r_2$ and $r$ are finite and the homomorphisms $P_1,P_2,A_1,A_2$ and $R$ can be represented by their corresponding matrices (\wrt chosen bases), denoted by $\mathbf{P_1}, \mathbf{P_2}, \mathbf{A_1}, \mathbf{A_2}$ and  $\mathbf{R}= \mathbf{P_1} \mathbf{A_1}-\mathbf{P_2}\mathbf{A_2}$, respectively (where $\mathbf{A_1}=\mathbf{A}_{H_1}$ and $\mathbf{A_2}=\mathbf{A}_{H_2}$ are the corresponding completion matrices). In this finitely generated case, all the conditions in \Cref{cor: 2-int fg} are algorithmically checkable and \Cref{thm: SIP FTA} follows.

\bigskip

Let us now extend the previous setup to finitely many arbitrary subgroups $H_1, \ldots, H_k\leqslant \FTA$ instead of two. In order to study the subgroup intersection $H_1\cap \cdots \cap H_k$, we first fix a basis for each subgroup, say 
 \begin{equation}
H_i =\langle u_{i,1}\t^\mathbf{a_{i,1}}, u_{i,2}\t^\mathbf{a_{i,2}}, \ldots ; \t^\mathbf{b_{i,1}}, \ldots ,\t^\mathbf{b_{i,s_i}}\rangle,
 \end{equation}
for every $i=1,\ldots,k$; here, the ranks $r_i=\rk(H_i\pi)$ and $r=\rk(H_1\pi\cap \cdots \cap H_k\pi)$ may be finite or infinite (again, with the dots above representing countably many elements in the infinite case). Let us consider the following notation, which is summarized in \Cref{fig: int pr} below. We denote:
\begin{enumerate}[ind]
\item the abelianization maps by $\rho_i \colon H_i\pi \onto \ZZ^{r_i}$, for $i=1,\ldots ,k$; also $\rho\colon \bigcap_{j=1}^k H_j\pi \onto \ZZ^{r}$; 

\item the inclusion maps by $\iota_i \colon \bigcap_{j=1}^k H_j\pi \into H_i\pi$, for $i=1,\ldots ,k$; 

\item the abelianization of the inclusion maps $\iota_i$ by $P_i \colon \ZZ^r \rightarrow \ZZ^{r_i}$, for $i=1,\ldots ,k$;

\item the completion homomorphisms by $A_i\colon \ZZ^{r_i}\to \ZZ^m$ (mapping the canonical $j$-th vector to $\mathbf{a_{i,j}}\in \ZZ^m$, for $i=1,\ldots ,k$ and $j\geq 1$); and  

\item the intersections of the subgroups $H_i$ with the free-abelian part as $L_i =H_i\cap \ZZ^m=\langle \mathbf{b_{i,1}},\ldots , \mathbf{b_{i,s_i}} \rangle$, where $0\leq s_i=\rk(L_i)\leq m$, and call $\mathbf{L_i}$ the $s_i \times m$ integral matrix having $\mathbf{b_{i,j}}$ as $j$-th row, $i=1,\ldots ,k$.
\end{enumerate}

\begin{figure}[H]
\centering
\begin{tikzcd}[row sep=25pt, column sep=25pt,ampersand replacement=\&]
 \big(\bigcap_{j=1}^{k}  H_j\big)\pi \,\normaleq\  
 \&[-30pt]\bigcap_{j=1}^{k}  H_j \pi
 \arrow[d,->>,"\rho"]
 \arrow[r,hook,"\iota_i"]
    \& H_i\pi
    \arrow[d,->>,"\rho_i"]
    \& (\forall i=1,\ldots,k)
    \\[10pt]
    \&\ZZ^{r}
    \arrow[ur,phantom,"///"]
    \arrow[r,->,"P_i"]\&
    \ZZ^{r_i} \arrow[ld,->,"A_i"]
    \\
    L_i \,\leqslant \hspace{-65pt} \&\Zm
   \end{tikzcd}
   \caption{The intersection diagram for $\bigcap_{j=1}^{k} H_j$}
   \label{fig: int pr}
\end{figure}

In the $k=2$ scenario, \Cref{prop: 2-int} crucially reduces the finite generability of the intersection of two subgroups of $\FTA$ to an abelian condition, which is easily decidable in the finitely generated case. The lemma below is a first necessary ingredient to generalize the aforementioned reduction to multiple intersections.

\begin{lem}\label{lem: nonempy multiple intersection}
Let $\mathbf{p_1},\ldots, \mathbf{p_k} \in \ZZ^m$, and $L_1,\ldots, L_k \leqslant \ZZ^{m}$. Then, the integral affine varieties $\mathbf{p_1}+L_1, \ldots ,\mathbf{p_k}+L_k$ intersect nontrivially (\ie $\bigcap_{j=1}^{k} (\mathbf{p_j}+L_j) \neq \varnothing$) if and only if 
 \begin{equation*}
\big( \mathbf{p_2}-\mathbf{p_1} \mid \mathbf{p_3}-\mathbf{p_2} \mid \cdots \mid \mathbf{p_k}-\mathbf{p_{k-1}}) \in \im (\mathbf{L}),
 \end{equation*}
where
 \begin{equation*}
\mathbf{L}=\left( \!\!\begin{array}{rrrcc} \mathbf{L_1} & & & & \\ -\mathbf{L_2} & \mathbf{L_2} & & & \\ & -\mathbf{L_3} & \mathbf{L_3} & & \\ & & \hspace{-5pt}\ddots & \hspace{-1pt} \ddots & \\ & & & -\mathbf{L_{k-1}} & \mathbf{L_{k-1}} \\ & & & & -\mathbf{L_{k}} \end{array}\right)\in M_{(\ssum
_j s_j )\times (k-1)m} (\ZZ)\,,
 \end{equation*}
and, for each $i=1,\ldots ,k$, $\mathbf{L_i}$ is a $s_i\times m$ integral matrix with row space $L_i$.
\end{lem}

\begin{proof} \phantom{\qedhere} \hspace{-.9cm}
Observe that $\bigcap_{j=1}^{k} (\mathbf{p_j}+L_j) \neq \varnothing$ if and only if $\mathbf{p_1}+\boldsymbol{\ell_1}=
\mathbf{p_2}+\boldsymbol{\ell_2}=\cdots =\mathbf{p_k}+\boldsymbol{\ell_k}$, for some $\boldsymbol{\ell_1}\in L_1$, $\boldsymbol{\ell_2}\in L_2$, $\ldots$, $\boldsymbol{\ell_k}\in L_k$. But this is the same as the existence of $\boldsymbol{\ell_1}\in L_1$, $\boldsymbol{\ell_2}\in L_2$, $\ldots$, $\boldsymbol{\ell_k}\in L_k$ such that $\mathbf{p_2}-\mathbf{p_1}= \boldsymbol{\ell_1}-\boldsymbol{\ell_2}$, $\mathbf{p_3}-\mathbf{p_2} =\boldsymbol{\ell_2}-\boldsymbol{\ell_3}$, $\ldots$, $\mathbf{p_{k}}-\mathbf{p_{k-1}} =\boldsymbol{\ell_{k-1}}-\boldsymbol{\ell_{k}}$. Or, equivalently, the existence of $\mathbf{a_1}\in \ZZ^{s_1}$, $\mathbf{a_2}\in \ZZ^{s_2}$, $\ldots$, $\mathbf{a_k}\in \ZZ^{s_k}$ such that
 \[
\left. \begin{array}{rcl} \mathbf{p_2}-\mathbf{p_1}\,=\,\mathbf{a_1}\mathbf{L_1}- \mathbf{a_2}\mathbf{L_2} & = & (\mathbf{a_1} \mid \mathbf{a_2}) \left( \begin{matrix} \phantom{-}\mathbf{L_1} \\ -\mathbf{L_2} \end{matrix} \right) \\ & \vdots & \\ \mathbf{p_k}-\mathbf{p_{k-1}}\,=\,\mathbf{a_{k-1}}\mathbf{L_{k-1}}- \mathbf{a_k}\mathbf{L_k} & = & (\mathbf{a_{k-1}} \mid \mathbf{a_k}) \left( \begin{matrix}  \mathbf{L_{k-1}} \\ -\mathbf{L_k} \end{matrix} \right) \end{array} \right\}
 \]
Putting these conditions together, this is the same as saying that there exists $\mathbf{a}\in \ZZ^{{\Sigma_j s_j}}$ such that
 \begin{equation}
(\mathbf{p_2}-\mathbf{p_1} \mid \mathbf{p_3}-\mathbf{p_2} \mid \cdots \mid \mathbf{p_k}-\mathbf{p_{k-1}}) =\mathbf{a} \mathbf{L}
\in \im (\mathbf{L}). \tag*{\qed}
 \end{equation}
\end{proof}

Let us now combine \Cref{lem: nonempy multiple intersection} with the characterization coming from the above diagram to obtain the multiple version of \Cref{prop: 2-int} and,  as a consequence, the decidability of the multiple subgroup intersection problem for free-abelian times free groups. 
In order to do so it is convenient to define the \defin{stack homomorphism} as
 \begin{equation} \label{eq: R def}
\begin{array}{rcl}
R=(R_2, \ldots ,R_k)\colon \ZZ^r & \to & \ZZ^m\oplus \stackrel{(k-1)}{\ldots} \oplus \ZZ^m=\ZZ^{(k-1)m} \\ w & \mapsto & (wR_2, \ldots ,wR_k),
\end{array}
 \end{equation}
where $R_i=P_iA_i-P_{i-1}A_{i-1}\colon$ $\ZZ^r \to \ZZ^m$, $i=2, \ldots ,k$. Note that, for the case $k=2$, $R=R_2 \colon \ZZ^r \to \ZZ^{(2-1)m}=\ZZ^m$, $w\mapsto w(P_2A_2-P_1A_1)$ agrees with the homomorphism $R$ used in \Cref{prop: 2-int} and \Cref{cor: 2-int fg}.

The introduced scheme and notation (see \Cref{fig: int pr} and \Cref{eq: R def}) together with \Cref{lem: nonempy multiple intersection} allows us to describe the projection $(H_1 \cap \cdots \cap H_k) \pi$ in a particularly convenient way, which unveils its character (and hence that of the intersection $H_1 \cap \cdots \cap H_k$).

\begin{thm}\label{prop: mint fg iff}
For any $k\geq 2$ subgroups $H_1, \ldots, H_k\leqslant \FTA$, $(H_1\cap \cdots \cap H_k)\pi=(\im (\mathbf{L})) R\preim\rho\preim$, which is a normal subgroup of $H_1\pi \cap \cdots \cap H_k\pi$. In particular, $H_1\cap \cdots \cap H_k$ is finitely generated if and only if one of the following conditions holds: 
 \begin{enumerate}[ind]
\item \label{item: r=0} $r=0,1$; 
\item \label{item: r} $2\leq r<\infty$ and $\rk \big((\im (\mathbf{L})) R\preim\big)=r$;
 \end{enumerate}
 where $r = \rk(H_1 \pi \cap \cdots \cap H_k \pi)$.
\end{thm}

\begin{proof}
Consider a basis for each of the given subgroups, $H_i =\langle u_{i,1}\t^\mathbf{a_{i,1}}, u_{i,2}\t^\mathbf{a_{i,2}}, \ldots ; \t^\mathbf{b_{i,1}}, \ldots ,\t^\mathbf{b_{i,s_i}}\rangle$, $i=1,\ldots ,k$, and keep all the notation introduced above (see Figure~\ref{fig: int pr} and \Cref{eq: R def}). Then,
 \begin{align*}
(H_1 \cap \cdots \cap H_k)\pi \,=\, & \big\{w \in
\textstyle{\bigcap_{j=1}^{k} H_j\pi \mid w\text{ has a common completion into $H_1, \ldots ,H_k$}} \big\} \\[2pt] \,=\  & \{w \in
\textstyle{\bigcap_{j=1}^{k} H_j\pi \mid (w\iota_1 \rho_1 A_1 +L_1) \cap \cdots \cap(w\iota_k \rho_k A_k +L_k) \neq \varnothing} \} \\[2pt] \,=\  & \{w \in
\textstyle{\bigcap_{j=1}^{k} H_j\pi \mid (w\rho P_1 A_1 +L_1) \cap \cdots \cap(w\rho P_k A_k +L_k) \neq \varnothing} \} 
\\[2pt]
\,=\  & \{w \in
\textstyle{\bigcap_{j=1}^{k} H_j\pi \mid w\rho R\in \im (\mathbf{L})} \} \\[2pt]
\,=\  & \big(\im (\mathbf{L})\big) R\preim \rho\preim,
 \end{align*}
where the second equality follows from \Cref{cor: compl computable}, the third one from the commutativity of the diagram in \Cref{fig: int pr}, and the fourth one from \Cref{lem: nonempy multiple intersection}. Note that, then, $(H_1 \cap \cdots \cap H_k)\pi$ is a normal subgroup of $H_1\pi \cap \cdots \cap H_k\pi$.

In particular, if $r=0,1$  (\ie if $H_1 \pi \cap \cdots \cap H_k \pi \isom \Free[r]$ is cyclic) then $(H_1 \cap \cdots \cap H_k)\pi$ is cyclic and hence $H_1 \cap \cdots \cap H_k$ is necessarily finitely generated. Otherwise (\ie if $2\leq r\leq \infty$), $(H_1 \cap \cdots \cap H_k)\pi$ contains the (nontrivial) commutator $[\Free[r], \Free[r]]$ and so, it is nontrivial. This implies that, when $r=\infty$, $(H_1 \cap \cdots \cap H_k )\pi$ is never finitely generated; whereas, for $2\leq r<\infty$,
 \begin{align*}
H_1 \cap \cdots \cap H_k \text{ is finitely generated}
&\,\Leftrightarrow\,
(H_1 \cap \cdots \cap H_k)\pi 
\text{ is finitely generated} \\
&\,\Leftrightarrow\,
\big(\im (\mathbf{L})\big) R\preim \rho\preim
\text{ is finitely generated}\\
&\,\Leftrightarrow\,
|\Free[r] : \big(\im (\mathbf{L})\big) R\preim \rho\preim| < \infty\\
&\,\Leftrightarrow\,
|\ZZ^r : \big(\im (\mathbf{L})\big) R\preim| < \infty
\\
&\,\Leftrightarrow\,
\rk((\im (\mathbf{L})) R\preim )=r,
 \end{align*}
where we have used 
\begin{enumerate*}[ind]
\item \Cref{rem: H fg iff Hpi fg}, in the first equivalence;
\item that $(H_1 \cap \cdots \cap H_k)\pi = \big(\im (\mathbf{L})\big) R\preim \rho\preim$, in the second equivalence;
\item that a nontrivial normal subgroup of a finitely generated free group is finitely generated if and only if it has finite index, in the third equivalence;
\item that the abelianization map $\rho\colon \Free[r] \onto \ZZ^r$ is surjective, in the fourth equivalence; and
\item that a subgroup of $\ZZ^r$ is of finite index if and only if it has maximum rank, in the fifth equivalence.
\end{enumerate*}
     
Putting together the obtained conditions, we reach the claimed result. 
\end{proof}

Finally, let's see that, as a consequence of \Cref{prop: mint fg iff}, when all the input subgroups are finitely generated (that is, when $r_1,\ldots ,r_k<\infty$ and so, $r<\infty$), we can decide algorithmically whether $H_1\cap \cdots \cap H_k$ is finitely generated and, in case it is, compute a basis for it. This generalizes \Cref{thm: SIP FTA} to multiple intersections.

\begin{thm}\label{thm: MSIP}
The Multiple Subgroup Intersection Problem for FTFA groups is computable. That is, there exists an algorithm which, on input 
a finite number of (finite sets of generators for) subgroups $H_1, \ldots, H_k\leqfg \FTA$, decides whether the intersection $H_1\cap \cdots \cap H_k$ is finitely generated and, in the affirmative case, computes a basis for it.
\end{thm}

\begin{proof} 
The result is obvious for $k=0,1$, so we assume $k \geq 2$. In \Cref{prop: mint fg iff}, we have reduced the finite generability of the intersection $H_1\cap \cdots \cap H_k$ down to explicit linear algebra conditions, which are clearly verifiable using standard techniques. Hence, the decision problem is computable.

Finally, let us assume that $H_1\cap \cdots \cap H_k$ is finitely generated, and let us see how to compute a basis for it. It is clear that if $(H_1\cap \cdots \cap H_k)\pi =\{1\}$ then we can easily compute a (free-abelian) basis for $H_1 \cap \cdots \cap H_k =L_1\cap \cdots \cap L_k\leqslant \ZZ^m$. 

Otherwise, $1\leq r=\rk (H_1\pi \cap \cdots \cap H_k\pi)=\rk \big( (\im (\mathbf{L})) \mathbf{R\preim}\big)<\infty$ and so, $|\ZZ^r : (\im (\mathbf{L})) \mathbf{R\preim}|=\ell<\infty$. 
With a standard linear algebra procedure, we can compute a set $\{\mathbf{c_1},\ldots ,\mathbf{c_\ell}\}$ of coset representatives of $\big(\im (\mathbf{L})\big) \mathbf{R\preim}$ in $\ZZ^r$, namely
 \begin{equation*}
\ZZ^r =\big( \mathbf{c_1}+(\im (\mathbf{L})) \mathbf{R\preim} \big)\sqcup \cdots \sqcup \big(\mathbf{c_\ell}+(\im (\mathbf{L})) \mathbf{R\preim}\big).
 \end{equation*}
Then, we can use pull-backs of Stallings' automata (see~\cite{kapovich_stallings_2002,delgado_list_2022}) to compute a free basis $\{v_1, \ldots, v_r\}$ for $H_1\pi \cap \cdots \cap H_k\pi$, and choose respective arbitrary $\rho$-preimages in $H_1\pi \cap \cdots \cap H_k\pi$, say $z_1,\ldots ,z_\ell$, of the vectors $\mathbf{c_1},\ldots ,\mathbf{c_\ell}$, in order to obtain a set of right coset representatives of $(H_1 \cap \cdots \cap H_k)\pi = (\im (\mathbf{L})) \mathbf{R\preim}\rho\preim$ in $H_1\pi \cap \cdots \cap  H_k\pi$:
 \begin{equation}\label{equ}
H_1\pi \cap \cdots \cap  H_k\pi=\big( (H_1 \cap \cdots \cap H_k)\pi \big)z_1 \sqcup \cdots \sqcup \big( (H_1 \cap \cdots \cap H_k)\pi \big) z_\ell.
 \end{equation}
Finally, we build the (finite) Schreier digraph for $(H_1 \cap \cdots \cap H_k)\pi \leqfi H_1\pi \cap \cdots \cap  H_k\pi$ with respect to $\{v_1, \ldots, v_r \}$ (which coincides with the Cayley graph of the finite abelian group 
 \[
\big( H_1\pi \cap \cdots \cap  H_k\pi \big)/ (H_1 \cap \cdots \cap H_k)\pi \simeq \ZZ^r/ (\im (\mathbf{L})) \mathbf{R\preim}
 \]
\wrt the set of generators $\{v_1\rho, \ldots ,v_r\rho\}$), and compute the desired basis for $H_1\cap \cdots \cap H_k$ in the following way:
\begin{enumerate}
\item Take the cosets $\big( (H_1 \cap \cdots \cap H_k)\pi \big)z_i =\big( (\im (\mathbf{L})) \mathbf{R\preim}\rho\preim \big) z_i$, $i=1,\ldots ,\ell$, from~\eqref{equ} as vertices, and with no edges.
\item For every vertex $\big( (\im (\mathbf{L})) \mathbf{R\preim}\rho\preim\big) z_i$, $i=1,\ldots ,\ell$, and every letter $v_j$, $j=1,\ldots ,r$, add an edge labeled $v_j$ from $\big( (\im (\mathbf{L})) \mathbf{R\preim}\rho\preim\big) z_i$ to the vertex $\big( (\im (\mathbf{L})) \mathbf{R\preim}\rho\preim\big) z_iv_j$; we can algorithmically identify this vertex among the available ones by repeatedly solving the membership problem for $\im (\mathbf{L})$, which is just solving a system of linear equations, i.e., 
 \begin{align*}
\big( (\im (\mathbf{L})) \mathbf{R\preim}\rho\preim\big) z_iv_j =\big( (\im (\mathbf{L})) \mathbf{R\preim}\rho\preim\big) z_{i'}
& \ \Leftrightarrow\  z_iv_jz_{i'}^{-1}\in (\im (\mathbf{L})) \mathbf{R\preim}\rho\preim \\
& \ \Leftrightarrow\  (z_iv_jz_{i'}^{-1}) \rho \mathbf{R}\in \im (\mathbf{L})\,.
 \end{align*}
Once we have run over all $i=1,\ldots ,\ell$ and all $j=1,\ldots ,r$, we have fully computed the finite Schreier graph $\Gamma:=\Gamma\big( (H_1 \cap \cdots \cap H_k)\pi,\, H_1\pi \cap \cdots \cap  H_k\pi,\,\{v_1, \ldots ,v_r\}\big)$, of the finite index subgroup $(H_1 \cap \cdots \cap H_k)\pi\leqfi H_1\pi \cap \cdots \cap  H_k\pi$ with respect to the free basis $\{v_1, \ldots ,v_r\}$ of $H_1\pi \cap \cdots \cap  H_k\pi$.
\item Select a maximal tree $T$ in $\Gamma$ and, for every edge $e\in E\Gamma\setminus ET$, read the label $u_e$ (a word on the $v_i$'s) of the closed path $T[\bp, \iota e]eT[\tau e, \bp]$, where $\bp$ denotes the vertex corresponding to the trivial coset $(H_1 \cap \cdots \cap H_k)\pi$, and $T[p,p']$ denotes the $T$-geodesic from vertex $p$ to vertex $p'$. It is well-known that the elements obtained in this way, say $\{u_1, \ldots ,u_q\}$, form a free basis for $(\im (\mathbf{L})) \mathbf{R\preim}\rho\preim =(H_1 \cap \cdots \cap H_k)\pi$.
\item Finally, solving $q$ (compatible) systems of linear equations, we can effectively compute vectors $\mathbf{a_1}, \ldots ,\mathbf{a_q}\in \ZZ^m$ such that $u_1\t^\mathbf{a_1}, \ldots ,u_q\t^\mathbf{a_q} \in H_1 \cap \cdots \cap H_k$ (see \Cref{lem: completions are cosets}); and a free-abelian basis $\{ \t^\mathbf{b_1}, \ldots ,\t^\mathbf{b_s}\}$ for $L_{H_1\cap \cdots \cap H_k}=L_1\cap \cdots \cap L_k$.
\end{enumerate}
By construction, the set $\{u_1\t^\mathbf{a_1}, \ldots ,u_q \t^\mathbf{a_q};\,  \t^\mathbf{b_1},\ldots ,\t^\mathbf{b_s}\}$ is the required basis for ${H_1 \cap \cdots \cap H_k}$.
\end{proof}

To finish this section, we prove~\Cref{thm: technical}, a result which will be of central importance for the arguments in the coming sections, but we think it is also of independent interest. We emphasize that it is a fairly sensitive result as slight modifications of it turn out to be false; see \Cref{ex: rk1,ex: false in FTFA}. 

\begin{prop}\label{prop: = for free}
Let $M', M''\leqslant \Free[n]$ be two subgroups of $\Free[n]$ in free factor position, \ie such that $\gen{M',M''}=M'*M''$. Then, for any $H'_1,\ldots ,H'_k\leqslant M'\leqslant \Free[n]$ and $H''_1,\ldots ,H''_k\leqslant M''\leqslant \Free[n]$, 
 \begin{equation}\label{eq: = for free}
\bigcap\nolimits_{j=1}^k \langle H'_j, H''_j\rangle
=
\Big\langle\, \bigcap\nolimits_{j=1}^k H'_j,\,\, \bigcap\nolimits_{j=1}^k H''_j \,\Big\rangle .
 \end{equation}
\end{prop}

\begin{proof}
For $k=1$ there is nothing to prove. And, by a straightforward induction on $k$, it is enough to prove the result for $k=2$, \ie
 \begin{equation} \label{eq: k=2}
\langle H'_1, H''_1\rangle \cap \langle H'_2, H''_2\rangle =\gen{ H'_1\cap H'_2,\, H''_1\cap H''_2}.
 \end{equation}
The inclusion to the left is obvious. For the converse, take an element $w\in \gen{H'_1,H''_1} \cap \gen{H'_2,H''_2}=(H'_1*H''_1 )\cap (H'_2*H''_2)$ and consider its normal form in these two free products, $w=u'_1u''_1\cdots u'_pu''_p$ with $u'_i\in H'_1$ and $u''_i\in H''_1$, and $w=v'_1v''_1\cdots v'_qv''_q$ with $v'_i\in H'_2$ and $v''_i\in H''_2$. Since $H'_1,H'_2\leqslant M'$ and $H''_1,H''_2\leqslant M''$, both are also valid normal forms for $w$ as an element from $M'*M''$. Hence, they must coincide: $p=q$, $u'_i=v'_i$, and $u''_i=v''_i$. This means that $u'_i=v'_i\in H'_1\cap H'_2$ and $u''_i=v''_i\in H''_1\cap H''_2$ and so, $w\in \langle H'_1\cap H'_2 ,\, H''_1\cap H''_2 \rangle$, as required. 

As an alternative proof, fix free bases $\mathcal{B}'$ and $\mathcal{B}''$ for $M'$ and $M''$, and it is clear that, for $i=1,2$, the Stallings automaton $\Gamma_i$ for $H'_i*H''_i$ w.r.t. the free basis $\mathcal{B}'\sqcup \mathcal{B}''$ of $M'*M''$ is just $\Gamma'_i*\Gamma''_i$, the disjoint union of the Stallings automaton $\Gamma'_i$ for $H'_i$ w.r.t. $\mathcal B'$ and the Stallings automaton $\Gamma''_i$ for $H''_i$ w.r.t. $\mathcal B''$, after identifying their basepoints (see~\cite{stallings_topology_1983,delgado_list_2022}). Since the labels at the edges in $\Gamma'_1$ and $\Gamma'_2$ (which belong to $\mathcal{B}'$), and in $\Gamma''_1$ and $\Gamma''_2$ (which belong to $\mathcal{B}''$) are completely disjoint, the pull-back of $\Gamma_1=\Gamma'_1*\Gamma''_1$ and $\Gamma_2=\Gamma'_2*\Gamma''_2$ will be the disjoint union of the pull-backs of $\Gamma'_1$ and $\Gamma'_2$, and of $\Gamma''_1$ and $\Gamma''_2$, after identifying their basepoints. Hence, $(H'_1*H''_1 )\cap (H'_2*H''_2)=(H'_1\cap H'_2)*(H''_1\cap H''_2)=\gen{ H'_1\cap H'_2,\, H''_1\cap H''_2}$, as required.
\end{proof}

In order to transfer some behaviors of free products (in $\Fn$) and direct sums (in $\Zm$) to $\Fn \times \Zm$, we introduce the terminology and notation below.

\begin{defn}
We say that two subgroups $M',M''$ of $\FTA$ are \defin{strongly complementary} if their projections to the free part are in free factor position and their projections to the free-abelian part are in direct sum position. If so, we say that $\gen{M',M''}$ is the \defin{strongly complementary product} (\defin{s.c.-product}) of $M'$ and $M''$, and we write $\gen{M',M''} =M'\cast M''$; that is,
 \[
\gen{M',M''} = M'\cast M'' \ \Leftrightarrow \ \left\{ \!\!
 \begin{array}{ll}
\gen{M'\pi,M''\pi}= M'\pi * M''\pi\leqslant \Free[n], \text{ and} \\ \gen{M'\tau,M''\tau}= M'\tau \oplus M''\tau\leqslant \ZZ^m.   
 \end{array}
\right.
 \]
\end{defn}

\begin{rem}
Note that, in general, $L_{M'}=M'\cap \ZZ^m \leqslant M'\tau$ and $L_{M''}=M''\cap \ZZ^m \leqslant M''\tau$ and so the condition $\gen{M'\tau, M''\tau}= M'\tau \oplus M''\tau$ implies (but it is stronger than) $\gen{L_{M'}, L_{M''}}= L_{M'} \oplus L_{M''}$. 
\end{rem}

Note also that, if $M'$ and $M''$ are strongly complementary, a basis for $M'\cast M''$ can be obtained by just taking the (disjoint) union of a basis for $M'$ and a basis for $M''$; in particular, $M'\cast M''$ is finitely generated if and only if $M'$ and $M''$ are so. Accordingly, we define the \defin{external s.c.-product} of two FATF groups
to be 
 \[(\Free[n'] \times \ZZ^{m'}) \cast (\Free[n''] \times \ZZ^{m''})
 \,=\,
 (\Free[n'] * \Free[n'']) \times (\ZZ^{m'} \oplus \ZZ^{m''})
 \,\isom\,
 \Free[n'+n''] \times \ZZ^{m'+ m''}.
 \]
In particular, s.c.-products agglutinate both free products and direct sums in the corresponding factors; that is, if $M',M''\leqslant \Fn$ then $M'\cast M''=M'*M''$; and if $M',M''\leqslant \Zm$ then $M'\cast M''=M'\oplus M''$.  

Also, given arbitrary groups $A,B,C,D$, and homomorphisms $\alpha \colon A \to C$ and $\beta\colon B \to D$, we define the three homomorphisms $\alpha *\beta\colon A*B\to C*D$, $\alpha \oplus \beta \colon A\oplus B\to C\oplus D$, and $\alpha \cast \beta \colon A*B\to C\oplus D$ in the natural ways, each mapping $a\in A$ to $a\alpha$, and $b\in B$ to $b\beta$ into the corresponding codomain.

\begin{thm}\label{thm: technical}
Let $H'_1,\ldots ,H'_k\leqslant \Free[n']\times \ZZ^{m'} = G'$ and $H''_1,\ldots ,H''_k\leqslant \Free[n'']\times \ZZ^{m''}=G''$ be $k\geq 2$ subgroups of $G'$ and $G''$, respectively. Write 
$r'=\rk\big( \bigcap_{j=1}^k H'_j\pi \big)$, $r''=\rk\big( \bigcap_{j=1}^k H''_j\pi \big)$, and consider $\gen{H'_1, H''_1},\ldots ,\gen{H'_k, H''_k} \leqslant G' \cast G'' = (\Free[n']*\Free[n''])\times (\ZZ^{m'}\oplus \ZZ^{m''}) $. Then, if $\min(r',r'')\neq 1$: 
 \begin{equation}\label{eq: technical}
\bigcap\nolimits_{j=1}^k \gen{H'_j, H''_j} \text{ is f.g.}
\ \Leftrightarrow \ 
\text{both } \bigcap\nolimits_{j=1}^k H'_j
\text{ \ and \ }
\bigcap\nolimits_{j=1}^k H''_j \text{ are f.g.}
 \end{equation}
\end{thm}

\begin{proof} 
Contrary to what happens in the free ambient (\eg in  \Cref{prop: = for free}), the claimed property does not seem to pass through induction over $k$. Hence, we need to use the precise description obtained in~\Cref{prop: mint fg iff} (instead of just using~\Cref{prop: 2-int} inductively) in order to analyze the multiple intersections $\bigcap_{j=1}^k H'_j$ and $\bigcap_{j=1}^k H''_j$, and to compare them to $\bigcap_{j=1}^k \langle H'_j, H''_j\rangle$.

Consider general subgroups $H'_1,\ldots,H'_k \leqslant G'$ and $H''_1,\ldots,H''_k \leqslant G''$ (with no extra assumptions, at the moment, on $r'$ and $r''$) and fix a basis for each of them, say:
 \begin{align*}
H'_i
&\,=\,
\big\langle u'_{i,1} \t^\mathbf{a'_{i,1}}, u'_{i,2} \t^\mathbf{a'_{i,2}},\ldots ; \t^\mathbf{b'_{i,1}},\ldots ,\t^\mathbf{b'_{i,s'_i}}\big\rangle \,\leqslant\, \Free[n']\times \ZZ^{m'}=G', \\
H''_i & \,=\, \big\langle u''_{i,1} \t^\mathbf{a''_{i,1}}, u''_{i,2} \t^\mathbf{a''_{i,2}},\ldots ; \t^\mathbf{b''_{i,1}},\ldots ,\t^\mathbf{b''_{i,s''_i}}\big\rangle \,\leqslant\, \Free[n'']\times \ZZ^{m''}=G'', 
 \end{align*}
where, for every $i=1,\ldots ,k$ and every $j$, $u'_{i,j}\in \Free[n']$, $u''_{i,j}\in \Free[n'']$, $\mathbf{a'_{i,j}},\, \mathbf{b'_{i,j}}\in \ZZ^{m'}$ and $\mathbf{a''_{i,j}},\, \mathbf{b''_{i,j}}\in \ZZ^{m''}$. For $i=1,\ldots ,k$, we write $r'_i=\rk(H'_i\pi)$, $r''_i=\rk(H''_i\pi)$, $s'_i=\rk(L'_i)\leq m'$, and $s''_i=\rk(L''_i)\leq m''$, where $L'_i=H'_i\cap \ZZ^{m'}$ and $L''_i=H''_i\cap \ZZ^{m''}$. 

With the data from the basis for $H'_i$, $i=1,\ldots ,k$, we adapt the intersection diagram (\Cref{fig: int pr}) and notation adding primes everywhere; \ie 
the abelianization maps are $\rho'_i \colon H'_i\pi \onto \ZZ^{r'_i}$ and $\rho'\colon \bigcap_{j=1}^k H'_j\pi \onto \ZZ^{r'}$\!;
the inclusion maps are $\iota'_i \colon \bigcap_{j=1}^k H'_j\pi \into H'_i\pi$;
the abelianization of the $\iota'_i$'s are ${P'_i \colon \ZZ^{r'} \rightarrow \ZZ^{r'_i}}$;
the completion homomorphisms are $A'_i\colon \ZZ^{r'_i}\to \ZZ^{m'}$; 
$L'_i =H'_i\cap \ZZ^{m'}=\langle \mathbf{b'_{i,1}},\ldots , \mathbf{b'_{i,s'_i}} \rangle$; $\mathbf{L'_i}$ is the $s'_i \times m'$ integral matrix having $\mathbf{b'_{i,j}}$ as $j$-th row, $i=1,\ldots ,k$; $\mathbf{L'}$ is the \smash{$(\sum_{j=1}^k s'_j )\times (k-1)m'$} integral matrix from \Cref{lem: nonempy multiple intersection}; and the stack homomorphism is 
 \begin{equation}
\begin{array}{rcl}
R'\colon \ZZ^{r'} & \to & \ZZ^{m'}\oplus \stackrel{_{(k-1)}}{\ldots} \oplus \ZZ^{m'} \\
w & \mapsto & (wR'_2,\ldots ,wR'_k) \, ,
\end{array}
 \end{equation}
where $R'_i=P'_iA'_i-P'_{i-1}A'_{i-1}$, for $i=2,\ldots ,k$.

In this situation, \Cref{prop: mint fg iff} characterizes when the intersection $H'_1\cap \cdots \cap H'_k$ is finitely generated: 
 \begin{equation}\label{fg'}
H'_1\cap \cdots \cap H'_k \quad \text{is f.g.} \quad \Leftrightarrow \quad \left\{ \begin{array}{l} r'=0,1, \quad \text{or} \\ 2\leq r'<\infty\quad \text{and}\quad (\im (\mathbf{L'}))(R'){\preim} \leqfi \ZZ^{r'}.
\end{array} \right.
 \end{equation}

In the exact same manner, we construct the intersection diagram for $H''_1\cap \cdots \cap H''_k$; the notation being exactly the same as in the previous paragraph, replacing primes by double primes everywhere. Repeating the previous argument, we have:
 \begin{equation}\label{fg''}
H''_1\cap \cdots \cap H''_k \quad \text{is f.g.} \quad \Leftrightarrow \quad \left\{ \begin{array}{l} r''=0,1, \quad \text{or} \\ 2\leq r''<\infty\quad \text{and}\quad (\im (\mathbf{L''}))(R''){\preim} \leqfi \ZZ^{r''}.
\end{array} \right.
 \end{equation}

Now, consider the combined subgroups $H_1=\langle H'_1, H''_1\rangle, \ldots ,H_k=\langle H'_k, H''_k\rangle$, all of them viewed as subgroups of $G =G'\cast G''=(\Free[n']*\Free[n''])\times (\ZZ^{m'}\oplus \ZZ^{m''})$. The key point is to observe that the intersection diagram for $H_1\cap \cdots \cap H_k$ is just the $\cast$-juxtaposition of the previous two (primed and double primed) diagrams. In fact, for $i=1,\ldots ,k$, since $H'_i$ and $H''_i$ are contained in the strongly complementary subgroups $\Free[n']\times \ZZ^{m'}$ and $\Free[n'']\times \ZZ^{m''}$ of $(\Free[n']*\Free[n''])\times (\ZZ^{m'}\oplus \ZZ^{m''})$ and hence they are strongly complementary as well, we can obtain a basis for $H_i=H'_i\cast H''_i$ by taking the disjoint union of the bases we already have for $H'_i$ and for $H''_i$, namely
 \begin{align*}
H_i=H'_i \cast  H''_i =
\big\langle
\{ u'_{i,1} \t^\mathbf{a'_{i,1}}, u'_{i,2} \t^\mathbf{a'_{i,2}},\ldots \}\cup\{ u''_{i,1} \t^\mathbf{a''_{i,1}}, u''_{i,2} \t^\mathbf{a''_{i,2}},\ldots \}; \t^\mathbf{b'_{i,1}},\ldots ,\t^\mathbf{b'_{i,s'_i}} , \t^\mathbf{b''_{i,1}},\ldots ,\t^\mathbf{b''_{i,s''_i}} \big\rangle, 
 \end{align*}
Now, by~\Cref{prop: = for free},
 \begin{align*}
H_1\pi \cap \cdots \cap H_k\pi
&\,=\,
\langle H'_1 \pi, H''_1\pi \rangle \cap \cdots \cap \langle H'_k \pi, H''_k \pi\rangle\\
&\,=\,
(H'_1 \pi \cap \cdots \cap H'_k\pi)*(H''_1 \pi \cap \cdots \cap  H''_k\pi).
 \end{align*}
Hence, the intersection diagram for  $H_1\pi \cap \cdots \cap H_k\pi$ consists of (for $i=1,\ldots ,k$):
\begin{enumerate}[ind]
\item the abelianization map
$\rho_i\colon H_i\pi
=
H'_i\pi *H''_i\pi \onto \ZZ^{r'_i}\oplus \ZZ^{r''_i}=\ZZ^{r_i}$ is $\rho_i=\rho'_i\cast \rho''_i$, where $r_i=r'_i+r''_i$; and
$\rho\colon \bigcap_{j=1}^{k} H_j\pi =(\bigcap_{j=1}^{k} H'_j \pi)*(\bigcap_{j=1}^{k} H''_j \pi) \onto \ZZ^{r'}\oplus \ZZ^{r''}=\ZZ^{r}$ is $\rho=\rho'\cast \rho''$, where $r=r'+r''$;

\item the natural inclusion map $\iota_i \colon  \bigcap_{j=1}^{k} H_j \pi = (\bigcap_{j=1}^{k} H'_j \pi)*(\bigcap_{j=1}^{k} H''_j \pi) \into H_i =H'_i\pi * H''_i\pi$ is $\iota_i =\iota'_i *\iota''_i$; 

\item the abelianization of $\iota_i$, $P_i\colon \ZZ^r =\ZZ^{r'}\oplus \ZZ^{r''} \to \ZZ^{r'_i}\oplus \ZZ^{r''_i}=\ZZ^{r_i}$, is $P_i=P'_i \oplus P''_i$; 

\item the completion homomorphism $A_i\colon \ZZ^{r_i}=\ZZ^{r'_i}\oplus \ZZ^{r''_i} \to \ZZ^{m'}\oplus \ZZ^{m''}=\ZZ^{m}$ is $A_i=A'_i \oplus A''_i$;

\item $L_i =H_i\cap (\ZZ^{m'}\oplus \ZZ^{m''})=\langle \mathbf{b'_{i,1}},\ldots ,\mathbf{b'_{i,s'_i}},\, \mathbf{b''_{i,1}},\ldots ,\mathbf{b''_{i,s''_i}} \rangle=L'_i\oplus L''_i$, where $s_i=\rk(L_i)= s'_i+s''_i\leq m'+m''=m$ and, accordingly, $\mathbf{L_i}$ is the $s_i\times m$ integral matrix $\mathbf{L_i} =\left(\!\begin{smallmatrix}
\mathbf{L'_i} & \\[-1pt] & \mathbf{L''_i} \end{smallmatrix}\!\right)$.
\end{enumerate}
\Cref{fig: int pr 2} summarizes the combined situation so far.
\begin{figure}[H]
\centering
\begin{tikzcd}[row sep=25pt, column sep=25pt,ampersand replacement=\&]
\&[-87pt]\big(\bigcap_{j=1}^{k} H_j\big)\pi \&[5pt] \&[-30pt]     \\[-27pt]
\&\rotatebox[origin=c]{270}{$\normaleq$}\\[-27]
\big(\bigcap_{j=1}^{k} H'_j \pi\big)*\big(\bigcap_{j=1}^{k} H''_j \pi\big)\,= \hspace{54pt}\& \bigcap_{j=1}^{k} H_j \pi \arrow[r, hook,"\iota_i=\iota'_i *\iota''_i"] \arrow[d, ->>,"\rho' \cast \rho'' = \rho"'] \& H_i\pi \arrow[d, ->>,"\rho_i =\rho'_i \cast \rho''_i"] \&[-30pt] =H'_i\pi *H''_i\pi \\
\ZZ^{r'}\oplus \ZZ^{r''}= \hspace{-52pt}\& \ZZ^r 
\arrow[ur, phantom, "///"]
\arrow[r,"P_i=P'_i\oplus P''_i"]  \& \ZZ^{r_i} \arrow[ld,"A_i=A'_i \oplus A''_i"]  \& \hspace{-5pt} =\ZZ^{r'_i}\oplus \ZZ^{r''_i} \\
  \ZZ^{m'} \oplus \ZZ^{m''} = \hspace{-45pt}  \& \ZZ^{m} 
\end{tikzcd}
\caption{The intersection diagram for $\bigcap_{j=1}^{k} \left( H'_j \cast H''_j \right)$}
\label{fig: int pr 2}
\end{figure}

Then, for $i=2,\ldots ,k$, the homomorphisms $R_i \colon \ZZ^{r'} \oplus \ZZ^{r''} \to \ZZ^{m'}\oplus \ZZ^{m''}$ are given by
 \begin{align*}
R_i & \,=\, P_iA_i-P_{i-1}A_{i-1} \\ & \,=\, (P'_i \oplus P''_i)(A'_i\oplus A''_i)-(P'_{i-1}\oplus P''_{i-1})(A'_{i-1}\oplus A''_{i-1}) \\ & \,=\, (P'_iA'_i \oplus P''_iA''_i)-(P'_{i-1}A'_{i-1}\oplus P''_{i-1}A''_{i-1}) \\ & \,=\, (P'_iA'_i-P'_{i-1}A'_{i-1})\oplus (P''_iA''_i-P''_{i-1}A''_{i-1}) \\ & \,=\, R'_i\oplus R''_i
 \end{align*}
and, using them, we build the combined stack homomorphism
 \begin{equation*}
\begin{array}{rcl}
R=(R_2,\ldots ,R_k) \colon \ZZ^r=\ZZ^{r'}\oplus \ZZ^{r''} & \to & (\ZZ^{m'} \oplus \ZZ^{m''})\, \oplus \stackrel{_{(k-1)}}{\ldots} \oplus \, (\ZZ^{m'} \oplus \ZZ^{m''}). \\
\mathbf{w} = (\mathbf{w'}, \mathbf{w''}) & \mapsto & (
\underbrace{\mathbf{w'}R'_2, \mathbf{w''}R''_2}_{\mathbf{w} R_2}, \ldots ,\underbrace{\mathbf{w'}R'_k, \mathbf{w''}R''_k}_{\mathbf{w}R_k})
\end{array}
 \end{equation*} 
Finally, the combined matrix $\mathbf{L}$ takes the form
  \begin{equation} \label{eq: combined L}
\mathbf{L}=\left( \!\!\begin{array}{rrrcc} \mathbf{L_1} & & & & \\ -\mathbf{L_2} & \mathbf{L_2} & & & \\ & -\mathbf{L_3} & \mathbf{L_3} & & \\ & & \hspace{-5pt}\ddots & \hspace{-1pt} \ddots & \\ & & & -\mathbf{L_{k-1}} & \mathbf{L_{k-1}} \\ & & & & -\mathbf{L_{k}} \end{array}\right)\in M_{(\ssum_j s_j )\times (k-1)m} (\ZZ)\, .
 \end{equation}
Now, applying \Cref{prop: mint fg iff} to the combined situation, we have that $(H_1\cap \cdots \cap H_k)\pi =(\im (\mathbf{L}))R\preim \rho\preim$ is a normal subgroup of $H_1\pi \cap \cdots \cap H_k\pi$ and, 
 \begin{equation}\label{fg}
H_1\cap \cdots \cap H_k \quad \text{is f.g.} \quad \Leftrightarrow \quad \left\{ \begin{array}{l} r=0,1, \quad \text{or} \\ 2\leq r<\infty\quad \text{and}\quad (\im (\mathbf{L}))R{\preim} \leqfi \ZZ^{r}.
\end{array} \right.
 \end{equation}
Furthermore, at this point we claim that
\begin{equation} \label{eq: subgroup =}
\big( \im (\mathbf{L}) \big)R\preim =(\im (\mathbf{L}'))(R')\preim \oplus (\im (\mathbf{L}''))(R'')\preim \leqslant \ZZ^{r'}\oplus \ZZ^{r''} \,=\, \ZZ^r.
\end{equation}
To see this, observe that $(\mathbf{w'}, \mathbf{w''})\in \big( \im (\mathbf{L}) \big)R\preim$ if and only if 
 \begin{equation} \label{eq: intercalat}
(\mathbf{w'}, \mathbf{w''})R
\,=\,
(\mathbf{w'}R'_2, \mathbf{w''}R''_2, \ldots ,\mathbf{w'}R'_k, \mathbf{w''}R''_k)
\,\in\,
\im (\mathbf{L}),
 \end{equation}
that is, if and only if there exist integral vectors $\mathbf{c'_1}, \mathbf{c''_1}, \ldots ,\mathbf{c'_k}, \mathbf{c''_k}$ of the corresponding sizes such that 
 \begin{equation}
(\mathbf{w'}R'_2, \mathbf{w''}R''_2, \ldots ,\mathbf{w'}R'_k, \mathbf{w''}R''_k)
\,=\,
(\mathbf{c'_1}, \mathbf{c''_1}, \ldots ,\mathbf{c'_k}, \mathbf{c''_k})\mathbf{L}.
 \end{equation}
But, due to the form \eqref{eq: combined L} of the matrix $\mathbf{L}$, this is equivalent to 
 \[
\bigg\{ \!
 \begin{array}{l}
(\mathbf{w'}R'_2, \ldots ,\mathbf{w'}R'_k)
\,=\,
(\mathbf{c'_1}, \ldots ,\mathbf{c'_k})\mathbf{L'} \text{,\quad and} \\[2pt]
(\mathbf{w''}R''_2, \ldots ,\mathbf{w''}R''_k)
\,=\,
(\mathbf{c''_1}, \ldots ,\mathbf{c''_k})\mathbf{L''},
 \end{array}
 \]
which is the same as saying that $\mathbf{w'}R'\in \im (\mathbf{L'})$ and $\mathbf{w''}R''\in \im (\mathbf{L''})$ (independently). Therefore, $(\mathbf{w'}, \mathbf{w''})\in \big( \im (\mathbf{L}) \big)R\preim$ if and only if $\mathbf{w'}\in \big( \im (\mathbf{L}') \big)(R')\preim$ and $\mathbf{w''}\in \big( \im (\mathbf{L}'') \big)(R'')\preim$, as claimed. 

Finally, let us make use of the hypothesis $\min(r',r'')\neq 1$; by symmetry, we can assume $1\neq r'\leq r''$. Taking into account that $r=r'+r''$, we prove~\eqref{eq: technical} by relating the equivalences~\eqref{fg'} and~\eqref{fg''} with~\eqref{fg}, and by distinguishing the different possibilities for $r'$ and $r''$: 
\begin{enumerate}
\item[(a)] if $r'= 0$ 
    (\ie if $\bigcap_{j=1}^k H'_j\pi= \Trivial$)
    then $r = r''$, $\bigcap_{j=1}^k H_j \pi= \bigcap_{j=1}^k H''_j \pi$, $\ZZ^r =\ZZ^{r''}$, $P_i=P''_i$, and $R_i=R''_i$; moreover,
    \Cref{eq: intercalat} takes the form
    $\mathbf{w''} R=(\mathbf{w''}R''_2, \ldots, \mathbf{w''}R''_k)\,\in\, \im (\mathbf{L''})$, and hence, $R=R''$. Therefore,
    $\big( \bigcap_{j=1}^{k} H_j \big) \pi =
    \im(\mathbf{L}) R\preim \rho\preim = \im(\mathbf{L''}) (R'')\preim (\rho'')\preim =
    \big( \bigcap_{j=1}^{k} H''_j \big) \pi$ and equivalence~\eqref{eq: technical} holds trivially. 
\item[(b)] if $r'\geq 2$ then $r\geq 2$ and, by~\eqref{eq: subgroup =}, $(\im (\mathbf{L}))R{\preim} \leqfi \ZZ^{r}$ if and only if $(\im (\mathbf{L'}))(R'){\preim} \leqfi \ZZ^{r'}$ and $(\im (\mathbf{L''}))(R''){\preim} \leqfi \ZZ^{r''}$. So, equivalence~\eqref{eq: technical} holds.
 
\end{enumerate}
This completes the proof. 
\end{proof}

\begin{rem} \label{rem: k=1}
In order to understand the intersection of $k\geq 2$ subgroups we have used \Cref{prop: mint fg iff}, with the technical hypothesis $\min(r',r'')\neq 1$ to avoid the exceptional behaviour shown in the example below, where the equivalence~\eqref{eq: technical} fails in the cases $r'=\rk\big( \bigcap_{j=1}^k H'_j\pi \big)=1$ and $r''=\rk\big( \bigcap_{j=1}^k H''_j\pi \big)\geq 1$. However, note that in the degenerated case $k=1$ there is no intersection to consider and the equivalence \eqref{eq: technical} follows immediately from $G'$ and $G''$ being strongly complementary in $G'\cast G''$, without any assumption on $r'$ or $r''$.
\end{rem}

\begin{exm}\label{ex: rk1}
Let $n\geq 1$, $G'=\Free[1]\times \ZZ^1=\langle x\mid - \rangle \times \langle t\mid - \rangle$ and $G''=\Free[n]\times \ZZ^0=\langle y_1,\ldots,y_{n} \mid -\rangle$, and take the subgroups $H'_1=\langle x\rangle$  and $H'_2=\langle x\t\rangle$ of $ G'$; and $H''_1 =H''_2 =\langle y_1,\ldots,y_{n}\rangle$ of~$G''$. Note that $r'=\rk(H'_1\pi \cap H'_2\pi)=1$ and $r''=\rk(H''_1\pi \cap H''_2\pi)=n\geq 1$. Consider then $H_1=\gen{H'_1, H''_1}=\gen{x,y_1,\ldots,y_{n}}$ and $H_2=\gen{H'_2, H''_2}=\gen{x t, y_1,\ldots,y_n}$, both as subgroups of $G'\cast G''=\langle x,y_1,\ldots,y_{n} \mid -\rangle\times \langle \t \mid - \rangle =\Free[n+1]\times \ZZ$.

Clearly, $H'_1 \cap H'_2=\{1\}$ and $H''_1\cap H''_2 =\gen{y_1,\ldots,y_{n}}$ are both finitely generated, whereas ${H_1 \cap H_2} =\langle x,y_1,\ldots,y_{n}\rangle \cap \langle x\t, y_1,\ldots,y_{n}\rangle =\{w(x,y_1,\ldots,y_{n}) \mid |w|_x =0\}=\normalcl{y_1,\ldots,y_{n}}_{\gen{x,y_1,\ldots,y_{n}}}$ is not. Hence, equivalence~\eqref{eq: technical} can fail whenever $\min(r',r'') = 1$ and \Cref{thm: technical} is as general as possible.

Moreover, observe that adding an extra freely independent letter, say $z$, to $G'$, $H'_1$ and $H'_2$ (and so, forcing $r'=2$ instead of $r'=1$) spoils the counterexample because then $H'_1 \cap H'_2$ becomes $\normalcl{z}_{\gen{x,z}}$, which is not finitely generated any more.
\end{exm}

It is worth mentioning that the proof of \Cref{thm: technical} is a bit delicate, in consonance with the fact that the statement is quite sensible to slight modifications. For example, an equality like \eqref{eq: = for free} (from which \Cref{thm: technical} would follow immediately) \emph{is not true} in the free-times-free-abelian case, \emph{even in the strongly complementary situation}; see the example below. This forced us to prove \Cref{thm: technical} directly, adding the complication that the statement does not pass well through induction on $k$ and so, forcing us to work with the multiple intersection situation (\Cref{prop: mint fg iff}) instead of the easier $k=2$ case (\Cref{prop: 2-int}).

\begin{exm} \label{ex: false in FTFA}
Consider in $\Free[4]\times \ZZ^2=\langle x_1,x_2,x_3,x_4 \mid -\rangle \times \langle t_1, t_2 \mid [t_1, t_2]\rangle$, the strongly complementary subgroups ${M'=\gen{x_1,x_2,\t^\mathbf{(1,0)}}}$ and $M''=\gen{x_3,x_4,\t^\mathbf{(0,1)}}$, and the respective subgroups
 \begin{align*}
H'_1 & =\gen{x_1,x_2}, \  H'_2=\gen{x_1\t^\mathbf{(1,0)}, x_2} \,\leqslant\, M', \\ H''_1 & =\gen{x_3,x_4}, \ H''_2=\gen{x_3\t^\mathbf{(0,1)}, x_4} \,\leqslant\, M'',
 \end{align*}
clearly satisfying $r'=\rk(H'_1\pi \cap H'_2\pi)=2$ and $r''=\rk(H''_1\pi \cap H''_2\pi)=2$. It is well known that $H'_1\cap H'_2=\langle x_1^{-i} x_2x_1^{i} ,\, i\in \ZZ \rangle$ and $H''_1\cap H''_2=\langle x_3^{-i} x_4 x_3^{i} ,\, i\in \ZZ \rangle$ are not finitely generated; see~\Cref{prop: Fn x Z^n no Howson}. Hence,
\[
\gen{H'_1 \cap H'_2,\, H''_1\cap H''_2}=(H'_1 \cap H'_2)*(H''_1\cap H''_2)=\langle x_1^{-i} x_2 x_1^{i},\, x_3^{-i} x_4 x_3^{i} \st i\in \ZZ \rangle
\]
and a quick look at its Stallings automaton tells us that this subgroup \emph{does not} contain the element $x_3^{-1} x_2 x_3$ while, clearly, $x_3^{-1} x_2x_3\in \gen{H'_1,\, H''_1}=\gen{x_1, x_2, x_3, x_4}$ and $x_3^{-1} x_2 x_3\in \gen{H'_2,\, H''_2}=\gen{x_1\t^\mathbf{(1,0)}, x_2, x_3\t^\mathbf{(0,1)}, x_4}$. Thus, \Cref{prop: = for free} is not true in the FTFA context, even in the strongly complementary situation. Note, however, that this example satisfies the less demanding statement from \Cref{thm: technical}, namely the subgroups $H'_1\cap H'_2$, $H''_1\cap H''_2$, and 
 \begin{align*}
\gen{H'_1,\, H''_1}\cap \gen{H'_2,\, H''_2} 
&\, =\, \gen{x_1, x_2, x_3, x_4} \cap \gen{x_1\t^\mathbf{(1,0)}, x_2, x_3\t^\mathbf{(0,1)}, x_4} \\
&\, =\, \{w(x_1, x_2, x_3, x_4) \mid |w|_{x_1}=|w|_{x_3}=0\}
 \end{align*}
are all of the same character: not finitely generated.
\end{exm}

\section{Intersection configurations}\label{sec: conf}

In this section we introduce some basic terminology, in order to conveniently state our next results. The first notion is that of $k$-configuration.

\begin{defn} \label{def: conf}
Let $k\in \NN_{\geq 1}$.
A \emph{$k$-configuration} is a map $\conf$ from ${\mathcal P}([k])\setminus \{\varnothing\}$ to the binary set~$\{\fg,\, \nfg\}$, i.e.,
 \begin{equation*}
\conf\colon {\mathcal P}([k])\setminus \{\varnothing\} \to \{\fg, \nfg\}.
 \end{equation*}
Note that this is just a fancy way to specify a family of nonempty subsets of $[k]$.
That is,
$\conf$ is just the indicator function 
 \smash{$\schi_{\mathcal{I}}^{_{(k)}}$}
of the inclusion $\mathcal{I} = (1)\conf\preim \subseteq \mathcal{P}([k])\setminus \{\varnothing\}$;
then we say that $\mathcal{I}$ is the \emph{support},
 and $k$ is the \defin{dimension}
of~$\conf$. 
In particular, the $k$-configuration $\zero^{_{(k)}}_{\phantom{a}}=\schi_{\varnothing}^{_{(k)}}$ (sending every nonempty set of indices to $\fg$) is called the \emph{zero $k$-configuration}. The \emph{one $k$-configuration} $\one^{_{(k)}}_{\phantom{a}}=\schi_{\parts}^{_{(k)}}$ is defined accordingly. When $\mathcal{I}=\{I\}$ (i.e., only one nonempty subset $I\subseteq [k]$ goes to~$1$), we write $\schi_{\{\hspace{-1pt}I\hspace{-1pt}\}}^{_{(k)}}=\schi_{I}^{_{(k)}}$ and we say that it is an \emph{almost-$\zero$} $k$-configuration. 
If the ambient is clear (or does not affect the ongoing argument) we usually omit any reference to it and simply write $\schi_{\mathcal{I}}^{_{(k)}}= \conf[\mathcal{I}]$, $\zero^{_{(k)}}_{\phantom{a}}=\zero$, $\one^{_{(k)}}_{\phantom{a}}=\one$, etc.
\end{defn}

A convenient way to visualize $k$-configurations is as 2-colored, oriented, \mbox{$k$-dimensional} hypercube digraphs, where white vertices correspond to nonempty subsets of $[k]$ going to $0$, black vertices correspond to nonempty subsets of $[k]$ going to $1$, and arcs (directed edges) denote inclusion into a subset with exactly one more element. In our representation, the top vertex corresponds to the empty subset (which is excluded in \Cref{def: conf} and will be ignored in the graphical representation), and the bottom vertex corresponds to the total subset $I=[k]$.

Note that, with this interpretation, the family of subsets containing (resp., not containing) a given index $i\in [1,k]$ is a maximal hyperface of the hypercube.

\begin{exm}
The lattice of subsets of $[k]$ with $k=3$ can be represented as
\begin{figure}[H]
    \centering
\begin{tikzpicture}[>=latex,
state/.style={circle,draw,minimum size=2mm}
]
\begin{scope}
    \newcommand{\dx}{0.75}
    \newcommand{\dy}{0.6}
    \node (000)  {$\varnothing$};
    
    \node (010) [below = \dy of 000] {$\{2\}$};
    \node (100) [left = \dx of 010] {$\{1\}$};
    \node (001) [right = \dx of 010] {$\{3\}$};
    
    \node (110) [below = \dy of 100] {$\{1,2\}$};
    \node (101) [below = \dy of 010] {$\{1,3\}$};
    \node (011) [below = \dy of 001] {$\{2,3\}$};
    
    \node (111) [below = \dy of 101] {$\{1,2,3\}$};
    
    \path[->] (000) edge (100);
    \path[->] (000) edge (010);
    \path[->] (000) edge (001);
    
    \path[->] (100) edge (110);
    \path[->] (100) edge (101);
    \path[->] (010) edge (110);
    \path[->] (010) edge (011);
    \path[->] (001) edge (101);
    \path[->] (001) edge (011);
    
    \path[->] (110) edge (111);
    \path[->] (101) edge (111);
    \path[->] (011) edge (111);
\end{scope}   
\end{tikzpicture}
\caption{The lattice of subsets of $[3] = \{1,2,3\}$}
\end{figure}
Now, coloring in white (resp., black) the vertices mapping to $0$ (resp., $1$) by $\conf$ provides the desired representation for the $3$-configurations (see \Cref{fig: confs}).
\begin{figure}[H]
    \centering
\begin{tikzpicture}[>=latex,
state/.style={circle,draw,minimum size=2mm}
]
\begin{scope}
    \newcommand{\dx}{0.6}
    \newcommand{\dy}{0.5}
    \node[smallstate] (000)  {};
    
    \node[state,fill=black] (010) [below = \dy of 000] {};
    \node[state] (100) [left = \dx of 010] {};
    \node[state,fill=black] (001) [right = \dx of 010] {};
    
    \node[state,fill=black] (110) [below = \dy of 100] {};
    \node[state] (101) [below = \dy of 010] {};
    \node[state] (011) [below = \dy of 001] {};
    
    \node[state,fill=black] (111) [below = \dy of 101] {};
    
    \node (l) [below = \dy/2 of 111] {$\conf[\{\{2\},\{3\},\{1,2\},\{1,2,3\}\}]$};
    
    \path[->] (000) edge (100);
    \path[->] (000) edge (010);
    \path[->] (000) edge (001);
    
    \path[->] (100) edge (110);
    \path[->] (100) edge (101);
    \path[->] (010) edge (110);
    \path[->] (010) edge (011);
    \path[->] (001) edge (101);
    \path[->] (001) edge (011);
    
    \path[->] (110) edge (111);
    \path[->] (101) edge (111);
    \path[->] (011) edge (111);
\end{scope}

\begin{scope}[xshift=111pt]
    \newcommand{\dx}{0.6}
    \newcommand{\dy}{0.5}
    \node[smallstate] (000)  {};
    
    \node[state] (010) [below = \dy of 000] {};
    \node[state] (100) [left = \dx of 010] {};
    \node[state] (001) [right = \dx of 010] {};
    
    \node[state] (110) [below = \dy of 100] {};
    \node[state] (101) [below = \dy of 010] {};
    \node[state] (011) [below = \dy of 001] {};
    
    \node[state] (111) [below = \dy of 101] {};
    
    \node (l) [below = \dy/2 of 111] {$\conf[\varnothing]$};
    
    \path[->] (000) edge (100);
    \path[->] (000) edge (010);
    \path[->] (000) edge (001);
    
    \path[->] (100) edge (110);
    \path[->] (100) edge (101);
    \path[->] (010) edge (110);
    \path[->] (010) edge (011);
    \path[->] (001) edge (101);
    \path[->] (001) edge (011);
    
    \path[->] (110) edge (111);
    \path[->] (101) edge (111);
    \path[->] (011) edge (111);
\end{scope}

\begin{scope}[xshift=222pt]
    \newcommand{\dx}{0.6}
    \newcommand{\dy}{0.5}
    \node[smallstate] (000)  {};
    
    \node[state] (010) [below = \dy of 000] {};
    \node[state,fill=black] (100) [left = \dx of 010] {};
    \node[state] (001) [right = \dx of 010] {};
    
    \node[state] (110) [below = \dy of 100] {};
    \node[state] (101) [below = \dy of 010] {};
    \node[state] (011) [below = \dy of 001] {};
    
    \node[state] (111) [below = \dy of 101] {};
    
    \node (l) [below = \dy/2 of 111] {$\conf[\set{1}]$};
    
    \path[->] (000) edge (100);
    \path[->] (000) edge (010);
    \path[->] (000) edge (001);
    
    \path[->] (100) edge (110);
    \path[->] (100) edge (101);
    \path[->] (010) edge (110);
    \path[->] (010) edge (011);
    \path[->] (001) edge (101);
    \path[->] (001) edge (011);
    
    \path[->] (110) edge (111);
    \path[->] (101) edge (111);
    \path[->] (011) edge (111);
\end{scope}

\begin{scope}[xshift=333pt]
    \newcommand{\dx}{0.6}
    \newcommand{\dy}{0.5}
    \node[smallstate] (000)  {};
    
    \node[state] (010) [below = \dy of 000] {};
    \node[state] (100) [left = \dx of 010] {};
    \node[state] (001) [right = \dx of 010] {};
    
    \node[state] (110) [below = \dy of 100] {};
    \node[state] (101) [below = \dy of 010] {};
    \node[state] (011) [below = \dy of 001] {};
    
    \node[state,fill=black] (111) [below = \dy of 101] {};
    
    \node (l) [below = \dy/2 of 111] {$\conf[\set{1,2,3}]$};
    
    \path[->] (000) edge (100);
    \path[->] (000) edge (010);
    \path[->] (000) edge (001);
    
    \path[->] (100) edge (110);
    \path[->] (100) edge (101);
    \path[->] (010) edge (110);
    \path[->] (010) edge (011);
    \path[->] (001) edge (101);
    \path[->] (001) edge (011);
    
    \path[->] (110) edge (111);
    \path[->] (101) edge (111);
    \path[->] (011) edge (111);
\end{scope}
\end{tikzpicture}
    \caption{From left to right, a random $3$-configuration, the $\zero$ $3$-configuration, and two almost-$\zero$ $3$-configurations}
    \label{fig: confs}
\end{figure}
\end{exm}

Now, we introduce a couple of ways to build new configurations from older ones.

\begin{defn}\label{sum of conf}
Given two $k$-configurations $\conf = \conf[\mathcal{I}],\, \conf'=\conf[\mathcal{I}']$, we define their \emph{join} as the new $k$-configuration 
$\conf[\mathcal{I}] \join \conf[\mathcal{I}'] = \conf[\mathcal{I} \cup \mathcal{I}']$. That is,
 \[
\begin{array}{rcl}
\conf\join \conf' \colon \mathcal{P}([k])\setminus \{\varnothing\} & \to & \{\fg,\, \nfg\} \\[2pt]
I & \mapsto & \left\{ \begin{array}{lcl} \fg & & \text{if } (I)\conf=(I)\conf'=\fg, \\ \nfg & & \text{otherwise.} \end{array}\right. \end{array}
 \]
\end{defn}

\begin{figure}[H]
\centering
\begin{tikzpicture}[>=latex,
state/.style={circle,draw,minimum size=3mm,inner sep=2pt}
]
\newcommand{\dx}{0.3}
\newcommand{\dy}{0.5}

\begin{scope}
    \node[smallstate] (00)  {};   
    \node (1) [below = \dy of 00] {};
    \node[state,fill=black] (10) [left = \dx of 1] {};
    \node[state] (01) [right = \dx of 1] {};    
    \node[state,fill=black] (11) [below = \dy of 1] {};
    \node [right = \dx of 01] {$\join$};
        
    \path[->] (00) edge (10);
    \path[->] (00) edge (01);
    \path[->] (01) edge (11);
    \path[->] (10) edge (11);
\end{scope}

\begin{scope}[xshift = 73pt]
    \node[smallstate] (00)  {};   
    \node (1) [below = \dy of 00] {};
    \node[state] (10) [left = \dx of 1] {};
    \node[state] (01) [right = \dx of 1] {};    
    \node[state,fill=black] (11) [below = \dy of 1] {};
    \node [right = \dx of 01] {$=$};
        
    \path[->] (00) edge (10);
    \path[->] (00) edge (01);
    \path[->] (01) edge (11);
    \path[->] (10) edge (11);
\end{scope}

\begin{scope}[xshift = 150pt]
    \node[smallstate] (00)  {};   
    \node (1) [below = \dy of 00] {};
    \node[state,fill=black] (10) [left = \dx of 1] {};
    \node[state] (01) [right = \dx of 1] {}; 
    \node[state,fill=black] (11) [below = \dy of 1] {};
        
    \path[->] (00) edge (10);
    \path[->] (00) edge (01);
    \path[->] (01) edge (11);
    \path[->] (10) edge (11);
\end{scope}

\end{tikzpicture}
\caption{A schematic representation for $\conf[\set{\set{1},\set{1,2}}] \join \conf[\set{1,2}] =\conf[\set{\set{1},\set{1,2}}]$}
\end{figure}

\begin{defn}
Given two $k$-configurations $\conf,\, \conf'$ and $\delta \in \{0,1\}$, we define the \emph{$\delta$-overlap} of $\conf$ and~$\conf'$ as the new $(k+1)$-configuration given by  
 \[
\begin{array}{rcl}
\conf\qplus_{\delta} \conf' \colon \mathcal{P}([k+1])\setminus \{\varnothing\} & \to & \{\fg,\, \nfg\} \\[2pt]
I\,\, & \mapsto & \Bigg\{ \begin{array}{lcl} (I)\conf & & \text{if } k+1\not\in I, \\ (I\setminus \{k+1\})\conf' & & \text{if } \{k+1\}\subsetneq I,\\
\delta & & \text{if } \{k+1\}=I.
\end{array} \end{array}
 \]
\end{defn}

\begin{figure}[H]
\centering
\begin{tikzpicture}[>=latex,
state/.style={circle,draw,minimum size=3mm,inner sep=2pt}
]
\newcommand{\dx}{0.3}
\newcommand{\dy}{0.5}
\begin{scope}
    \node[smallstate] (00)  {};   
    \node (1) [below = \dy of 00] {};
    \node[state,fill=black] (10) [left = \dx of 1] {};
    \node[state] (01) [right = \dx of 1] {};    
    \node[state,fill=black] (11) [below = \dy of 1] {};
    \node [right = \dx of 01] {$\qplus_1$};
    \path[->] (00) edge (10);
    \path[->] (00) edge (01);
    \path[->] (01) edge (11);
    \path[->] (10) edge (11);
\end{scope}

\begin{scope}[xshift = 75pt]
    \node[smallstate] (00)  {};   
    \node (1) [below = \dy of 00] {};
    \node[state] (10) [left = \dx of 1] {};
    \node[state] (01) [right = \dx of 1] {};    
    \node[state,fill=black] (11) [below = \dy of 1] {};
    \node [right = \dx of 01] {$=$};
    \path[->] (00) edge (10);
    \path[->] (00) edge (01);
    \path[->] (01) edge (11);
    \path[->] (10) edge (11);
\end{scope}

\begin{scope} [xshift=155pt]
    \renewcommand{\dx}{0.5}
    \renewcommand{\dy}{0.5}
    \node (0) {};
    \node[smallstate] [above = \dy/2 of 0] (000)  {};
    \node[state] (010) [below = \dy of 000] {};
    \node[state,fill=black] (100) [left = \dx of 010] {};
    \node[state,fill=black] (001) [right = \dx of 010] {};
    \node[state,fill=black] (110) [below = \dy of 100] {};
    \node[state] (101) [below = \dy of 010] {};
    \node[state] (011) [below = \dy of 001] {};
    \node[state,fill=black] (111) [below = \dy of 101] {};
    \path[->] (000) edge (100);
    \path[->] (000) edge (010);
    \path[->] (000) edge (001);
    \path[->] (100) edge (110);
    \path[->] (100) edge (101);
    \path[->] (010) edge (110);
    \path[->] (010) edge (011);
    \path[->] (001) edge (101);
    \path[->] (001) edge (011);
    \path[->] (110) edge (111);
    \path[->] (101) edge (111);
    \path[->] (011) edge (111);
\end{scope}

\end{tikzpicture}
\caption{A schematic representation for $\conf[\set{\set{1},\{1,2\}}] \qplus_1 \conf[\set{\set{1,2}}] = \conf[\set{\set{1},\set{3},\set{1,2},\set{1,2,3}}]$}
\end{figure}

\begin{rem}
Note that the join of two $k$-configurations is, again, a $k$-configuration whereas their $\delta$-overlap is a $(k+1)$-configuration. Moreover, the join is a commutative operator while, in general, the $\delta$-overlap is not.
\end{rem}

\medskip

Next, we define the related notion of intersection configuration, which is essential to state our results.

\begin{defn}\label{def: intersection configuration 2}
Let $G$ be a group, let $k\geq 1$, and let $\mathcal{H}=\{ H_i \}_{i\in [k]}$ be a family of $k$ subgroups of~$G$ (with possible repetitions). For every nonempty $I\subseteq [k]$, we write $H_I =\bigcap_{\, i\in I} H_i$; note that $H_I \cap H_J = H_{I \cup J}$. We define the \emph{intersection configuration determined by $\mathcal{H}$}, denoted by $\conf^\mathcal{H}$,
as 
\[
\begin{array}{rcl} \conf^\mathcal{H} \colon \parts & \to & \{\fg,\, \nfg\} \\[2pt]
I & \mapsto & \left\{ \begin{array}{ll} \fg & \text{if } H_I \text{ is finitely generated,} \\ \nfg & \text{if } H_I \text{ is not finitely generated.} \end{array}\right. \end{array}
 \]
We say that a $k$-configuration $\conf$ is \emph{realizable in $G$} if it is the intersection configuration of some family of $k$ subgroups of $G$; that is, if there exists a family $\mathcal{H}$ of $k$ subgroups of $G$ such that $\conf^{\mathcal H}=\conf$; in this case, we also say that ${\mathcal H}$ \emph{realizes} $\conf$, and that $G$ \emph{admits a realization} of $\conf$.
\end{defn}

For example, the $\zero$ $k$-configuration is always realizable in any group $G$ (realized, for instance, by the trivial $k$-family ${\mathcal H}=\{ \{1\} \}_{i\in [k]}$). On the other hand, the $\one$ $k$-configuration is realizable in a group~$G$ if and only if $G$ contains a non-f.g. subgroup $H\leqslant G$; in this case, it is enough to take ${\mathcal H}=\{ H \}_{i\in [k]}$. As a third example, for a $k$-configuration $\conf$ to be realizable in a free group $\Fn$, a necessary condition is that it does not violate the Howson property, i.e., for every nonempty $ I,J\subseteq [k]$, $(I)\conf=(J)\conf=\fg$ implies $(I\cup J)\conf=\fg$. That is, two white vertices never \emph{meet} at a black vertex in the directed 2-colored hypercube representing $\conf$. In Section~\ref{sec: free} we see that this condition indeed characterizes the realizable configurations in a free group $\Fn$, $n\geq 2$ (see Theorem~\ref{thm: char Fn}).

Of course, if $G_1\leqslant G_2$ then every $k$-configuration realizable in $G_1$ is also realizable in $G_2$. In particular since, for $2\leq r_1, r_2\leq \infty$, $\Free[r_1]\times \ZZ^m$ and $\Free[r_2]\times \ZZ^m$ are both subgroups of each other, the $k$-configurations realizable in $\Free[r_1]\times \ZZ^m$ coincide with those realizable in $\Free[r_2]\times \ZZ^m$; therefore, when considering FTFA groups, it is enough to restrict our attention to $\Free[2]\times \ZZ^m$. In contrast, as we see below, the abelian rank $m$ plays an important role in this respect: the set of $k$-configurations which are realizable in $\Free[2]\times \ZZ^m$ grows strictly with $m$ (see~\Cref{prop: almost0}).

A natural question in this context is whether there exist groups admitting realizations of \emph{any} finite configuration.

\begin{defn} \label{def: intersection-saturated}
A group $G$ is called \defin{intersection-saturated} if every finite configuration is realizable in~$G$.
\end{defn}

In \Cref{sec: realizing} we use the results obtained for FTFA groups to  exhibit explicit examples of finitely presented intersection-saturated groups.

\section{Unrealizable configurations}\label{sec: unrealizable}

The description obtained in~\cite{delgado_stallings_2022} suggests a high degree of flexibility for the intersections of subgroups of FTFA groups: not only these groups are non-Howson, but it is not possible to bound the rank of the intersection of two finitely generated subgroups in terms of their ranks, even when it is finitely generated. In this section we show that, despite this flexibility, there are indeed obstructions to the realizability of $k$-configurations in $\Free[n]\times \ZZ^m$, and that these obstructions are dictated by the ambient abelian rank~$m$. The cornerstone result is the following easy lemma.

\begin{lem}\label{lem: infind}
Let $H_1,\ldots ,H_k $ be $k\geq 2$ arbitrary subgroups of $ \Fn\times \ZZ^m$. If, for some nonempty subsets $I,J\subseteq [k]$, $H_I = \bigcap_{i \in I} H_i$ and $H_J= \bigcap_{j \in J} H_j$ are finitely generated whereas $H_{I\cup J}=H_I\cap H_J$ is not, then there exist $i\in I$ and $j\in J$ such that $L_i = H_i \cap \ZZ^m$ and $L_j = H_j \cap \ZZ^m$ both have infinite index in $\ZZ^m$; that is, $L_i$ and $L_j$ are free-abelian groups of rank strictly smaller than $m$.
\end{lem}

\begin{proof}
Consider the intersection diagram for the subgroups $H_I$ and $H_J$ (see Figure~\ref{fig: int 2}, replacing $1,2$ with $I,J$, respectively). Since $H_I$ and $H_J$ are finitely generated, $H_I\pi$ and $H_J\pi$ are so and, by the Howson property of free groups, $H_I\pi \cap H_J\pi$ is finitely generated as well. On the other hand, $H_{I\cup J} = H_I \cap H_J$ (and hence $(H_I\cap H_J)\pi=(L_I+L_J)R\preim\rho\preim$, which is a normal subgroup of $H_I \pi \cap H_J\pi$, see \Cref{prop: 2-int}) is not finitely generated. Therefore, $(H_{I}\cap H_{J})\pi$ must have infinite index in $H_I \pi \cap H_J\pi$, and $L_I+L_J$ (and so both $L_I$ and $L_J$) must also have infinite index in $\ZZ^m$. Since $L_I=\bigcap_{i\in I} L_i$, at least one of the $L_i$'s, $i\in I$, must have infinite index in $\ZZ^m$, i.e., $\rk(L_i)\leq m-1$; similarly for $J$.  
\end{proof}

\begin{rem} \label{rem: independence}
We say that a collection of $r$ subsets $\mathcal{I}=\{I_1, \ldots ,I_r\}\subseteq {\mathcal P}([k])$ is \emph{independent} if, when considering all the $2^r$ possible unions among them, we obtain $2^r$ different results, i.e., whenever for every $S,T\subseteq \{1,\ldots ,r\}$, $\bigcup_{s\in S} I_s =\bigcup_{t\in T} I_t$ implies $S=T$ (understanding that $\bigcup_{s\in \emptyset} I_s =\emptyset$). Note the following immediate properties of this notion: (i) any collection of sets containing the empty set $\varnothing$ is not independent; (ii) for $I\subseteq [k]$, $\{ I\}$ is independent if and only if $I\neq \varnothing$; (iii) if $\{I_1, \ldots ,I_r\}$ is independent then $r\leq k$; (iv) if $\{I_1, \ldots ,I_r\}$ is independent then any subset of it is also independent. Moreover, avoiding coincidences with the total union is enough to get independence: $\{I_1, \ldots ,I_r\}$ is independent if and only if, for every $M\subseteq \{1,\ldots ,r\}$ with $|M|=r-1$, we have $\bigcup_{m\in M} I_m \neq I_1\cup \cdots \cup I_r$. Indeed, the implication to the right is obvious; for the implication to the left, take $S,T\subseteq \{1,\ldots ,r\}$ satisfying $\bigcup_{s\in S} I_s =\bigcup_{t\in T} I_t$; if $S\neq T$ then there is an index contained in one of them and not in the other, say $j\in S\setminus T$, and then\footnote{
We use the standard notation putting a hat to denote a missing element.
}
 \begin{align*}
I_1\cup I_2\cup \cdots \cup I_r
&\,=\,
\textstyle{I_1\cup \cdots \cup \widehat{I_j} \cup \cdots \cup I_r\cup \bigcup_{s\in S} I_s} \\
&\,=\,
\textstyle{I_1\cup \cdots \cup \widehat{I_j}\cup \cdots \cup I_r \cup \bigcup_{t\in T} I_t
\,=\, I_1\cup \cdots \cup \widehat{I_j}\cup \cdots \cup I_r,}
 \end{align*}
contradicting our assumption.
\end{rem}

\begin{prop}\label{prop: main obstruction}
Let $\conf$ be a $k$-configuration for which there is a (independent) collection of $r\geq 2$ nonempty subsets $I_1,\ldots ,I_r\subseteq [k]$ such that, for every $j\in \{1,\ldots ,r\}$, $(I_1\cup \cdots \cup \widehat{I_j}\cup\cdots \cup I_r)\conf=\fg$, but $(I_1\cup \cdots \cup I_r)\conf=\nfg$. Then $\conf$ is not realizable in $\Fn\times \ZZ^{r-2}$.
\end{prop}

\begin{proof}
Note that, by \Cref{rem: independence}, the hypothesis on $\conf$ forces the collection of subsets $\{ I_1,\ldots ,I_r\}$ to be independent; in particular, they are all nonempty and $2\leq r\leq k$.

We prove the non realizability of $\conf$ in $\Fn\times \ZZ^{r-2}$, by induction on $r\geq 2$. In the case $r=2$, the statement is clearly true, as otherwise the Howson property for free groups $\Fn =\Fn\times \ZZ^{2-2}$ would be violated.

Now, assume the claim true for $r-1 \geq 2$, and let us prove it for $r$. Let $\conf$ be a $k$-configuration with $I_1,\ldots ,I_r$ satisfying the hypothesis, assume it is realizable in $\Fn\times \ZZ^{r-2}$, say by subgroups $H_1,\ldots ,H_k\leqslant \Fn\times \ZZ^{r-2}$, and let us find a contradiction. We have that, for every $j\in \{1,\ldots ,r\}$, $H_{I_1\cup \cdots \cup \widehat{I_j}\cup\cdots \cup I_r}= H_{I_1}\cap \cdots \cap \smash{\widehat{H_{I_j}}}\cap \cdots \cap H_{I_r}$ is finitely generated, while $H_{I_1\cup \cdots \cup I_r}=H_{I_1}\cap \cdots \cap H_{I_r}$ is not. Since both $H_{I_2\cup \cdots \cup I_r}$ and $H_{I_1\cup I_3\cup \cdots\cup I_r}$ are finitely generated but their intersection $
H_{I_1\cup \cdots \cup I_r}$ is not, \Cref{lem: infind} tells us that there exists $\ell \in I_1\cup \cdots \cup I_r$ such that $L_{\ell}$ has infinite index in $\ZZ^{r-2}$ and so, rank less than or equal to $r-3$. Up to renumbering the subsets, we can assume $\ell\in I_r$ and so, $L_{I_r}=\bigcap_{i\in I_r} L_i$ also has rank less than or equal to $r-3$. Note that, as a group, $H_{I_r}$ is then isomorphic to the direct product of a free group $F$, and $L_{I_r}\isom \ZZ^{\rk (L_{I_r})}\leqslant \ZZ^{r-3}$. 

Consider now $H'_i=H_i\cap H_{I_r}$, for $i=1,\ldots ,k$. On one hand, these are all subgroups of $H_{I_r}\isom F\times \ZZ^{\rk (L_{I_r})}\leqslant \Free[n]\times \mathbb{Z}^{r-3}$. On the other hand, the (independent) collection of subsets $I_1,\ldots ,I_{r-1}\subseteq [k]$ satisfy that, for every $j\in \{1,\ldots ,r-1\}$,
 \begin{align*}
H'_{I_1\cup \cdots \cup \widehat{I_j}\cup\cdots \cup I_{r-1}}
&\,=\, H'_{I_1}\cap \cdots \cap \widehat{H'_{I_j}}\cap \cdots \cap H'_{I_{r-1}}\\
&\,=\, (H_{I_1}\cap H_{I_r})\cap \cdots \cap \reallywidehat{(H_{I_j}\cap H_{I_r})}\cap \cdots \cap (H_{I_{r-1}}\cap H_{I_r})\\
&\,=\, H_{I_1}\cap \cdots \cap \reallywidehat{H_{I_j}}\cap \cdots \cap H_{I_r}
\,=\, H_{I_1\cup \cdots \cup \widehat{I_j}\cup\cdots \cup I_r}
\end{align*}
is finitely generated, whereas
 \begin{equation*}
H'_{I_1\cup \cdots \cup I_{r-1}}
\,=\, H'_{I_1}\cap \cdots \cap H'_{I_{r-1}}= H_{I_1}\cap \cdots \cap H_{I_{r-1}}\cap H_{I_r}=H_{I_1\cup \cdots \cup I_r}
 \end{equation*}
is not. This contradicts the inductive hypothesis.  
\end{proof}

\begin{exm}
This last result shows explicit restrictions in the lattice of subgroups of $\FTA$. For example, for $k=r=3$, it is telling us the following: if $H_1,H_2,H_3\leqslant \Fn\times \ZZ$ are arbitrary subgroups, and $H_1\cap H_2$, $H_1\cap H_3$, and $H_2\cap H_3$ are all finitely generated, then $H_1\cap H_2\cap H_3$ must be finitely generated as well. 
\end{exm}

The proposition below shows that, by strictly increasing the abelian rank $m$, the set of configurations realizable in $\FTA$ strictly increases as well. 

\begin{prop}\label{prop: almost0}
The $k$-configuration $\conf[{[k]}]$ is realizable in $\Free[2] \times \ZZ^{k-1}$, but not in $\Free[2] \times \ZZ^{k-2}$. 
\end{prop}

\begin{proof}
The second claim follows from \Cref{prop: main obstruction}, 
taking $r=k$ and $I_1 = \set{1}, \ldots , I_k=\set{k}$.

Let us prove the first claim. For $k=1$ the statement is just saying that the $1$-configuration $\{1\}\mapsto \nfg$, namely $\one$, is realizable in $\Free[2]\times \ZZ^{1-1}=\Free[2]$. This is obviously true since it is enough to take $H_1\leqslant \Free[2]$ to be any non-f.g. subgroup. 

Assume $k\geqslant 2$. We need to construct a family of subgroups $\mathcal{H} = \{ H_1,\ldots ,H_k\}$ of $\Free[2] \times \ZZ^{k-1}$ such that all partial intersections $H_I$  (where $\varnothing \neq I\subsetneq [k]$) are finitely generated, while the total one $H_{[k]}$ is not. Let $\{x,y\}$ be two free letters generating $\Free[2]$, let $\{\mathbf{e_1}, \ldots ,\mathbf{e_{k-1}}\}$ be the canonical free-abelian basis for $\ZZ^{k-1}$, and consider the following subgroups:
 \begin{align*}
H_1 &\,=\,  \langle x, y; \t^\mathbf{e_2}, \ldots, \t^\mathbf{e_{k-1}} \rangle \,\leqslant\, \Free[2] \times \ZZ^{k-1}, \\
H_2 &\,=\, \langle x, y; \t^\mathbf{e_1}, \t^\mathbf{e_3}, \ldots, \t^\mathbf{e_{k-1}} \rangle \,\leqslant\, \Free[2] \times \ZZ^{k-1}, \\
&\hspace{7pt} \vdots \\
H_{k-1} &\,=\, \langle x, y; \t^\mathbf{e_1}, \ldots, \t^\mathbf{e_{k-2}} \rangle \,\leqslant\, \Free[2] \times \ZZ^{k-1}, \\
H_k &\,=\, \langle x, y\t^\mathbf{e_1}; \t^\mathbf{e_2-e_1}, \ldots, \t^\mathbf{e_{k-1}-e_1} \rangle =\langle x, y\t^\mathbf{e_1}, \ldots, y\t^\mathbf{e_{k-1}} \rangle \,\leqslant\, \Free[2] \times \ZZ^{k-1}.
 \end{align*}
For a given nonempty set of indices $I\subseteq [k]$, let us compute $H_I$ by distinguishing the following three possible cases:

\emph{Case 1: $k\notin I$.} In this case, clearly, $H_I=\langle x, y; \t^\mathbf{e_j} \text{ for } j\notin I \rangle$, which is finitely generated.

\emph{Case 2: $k\in I\subsetneq [k]$.} In this case, without loss of generality we may assume that $1\notin I$ and so, $H_I$ is again finitely generated:
 \begin{align*}
H_I=H_{I\setminus \{k\}} \cap H_k & =\langle x, y; \t^\mathbf{e_j} \text{ for } j\not\in I \rangle \cap \langle x, y\t^\mathbf{e_1}, y\t^\mathbf{e_2}, \ldots ,y\t^\mathbf{e_{k-1}} \rangle \\
& =
\{w(x,y)\t^{\mathbf{a}} \mid a_j=0, \forall j\in I \setminus \{k\} \} \cap 
\{w(x,y)\t^{\mathbf{a}} \mid a_1+\cdots +a_{k-1}=|w|_y\} \\
& =
\{w(x,y)\t^{\mathbf{a}} \mid a_1+\cdots +a_{k-1}=|w|_y \text{ and } a_j=0\,\, \forall j\in I \setminus \{k\}\} \\
& =
\langle x, y\t^\mathbf{e_j} \text{ for } j\not\in I \rangle \\
& =
\langle x, y\t^{\mathbf{e_1}}; \t^\mathbf{e_j-e_1} \text{ for } j\not\in I \rangle.
 \end{align*}

\emph{Case 3: $I=[k]$}. In this case, 
 \begin{equation*}
H_I=(H_1\cap \cdots \cap H_{k-1})\cap H_k =\langle x, y \rangle \cap \langle x, y\t^\mathbf{e_1}; \t^\mathbf{e_2-e_1}, \ldots ,\t^\mathbf{e_{k-1}-e_1}\rangle =  \normalcl{x}_{\Free[2]},
 \end{equation*}
the normal closure of $x$ in $\Free[2]$, which is not finitely generated. 
\end{proof}

Apart from $\conf[{[k]}]$ being an explicit example of a $k$-configuration which is realizable in $\Free[2]\times \ZZ^{k-1}$ but not in $\Free[2]\times \ZZ^{k-2}$, we observe that there are strong restrictions which \emph{every} realization of $\conf[{[k]}]$ in $\Free[2]\times \ZZ^{k-1}$ must satisfy. This supports the idea that finding an exact characterization of the $k$-configurations realizable in $\Free[2] \times \ZZ^{m}$ for a given value of $m$, may be a complicated task.  

\begin{cor}
Every realization of the $k$-configuration $\conf[{[k]}]$ in $\Free[2] \times \ZZ^{k-1}$ must mandatorily be with subgroups $H_1,\ldots ,H_k\leqslant \Free[2] \times \ZZ^{k-1}$ satisfying $\rk (L_{H_i})\geq k-2$, $i=1,\ldots ,k$.
\end{cor}

\begin{proof}
Suppose that a certain family $\mathcal{H}=\{H_1,\ldots ,H_k\}$ of subgroups of $\Free[2] \times \ZZ^{k-1}$ realizes $\conf[{[k]}]$.
Then, clearly, the new family $\{H_1\cap H_i, \ldots , H_{i-1}\cap H_i, H_{i+1}\cap H_i, \ldots ,H_k\cap H_i\}$ realizes $\conf[{[k-1]}]$. 
Since, for all $j=1,\ldots ,\widehat{i}, \ldots ,k$, $H_j\cap H_i\leqslant H_i \isom \Free[r]\times L_{H_i}\isom \Free[r]\times \mathbb{Z}^{\rk (L_{H_i})}$ for some $0\leq r\leq \infty$, \Cref{prop: almost0} tells us that $\rk (L_{H_i})\geq k-2$, for $i=1,\ldots ,k$.
\end{proof}

\begin{rem}\label{prop: must}
Any realization of a $k$-configuration $\conf[\mathcal{I}]$ in $\Free[n] \times \ZZ^{m}$ must be by subgroups $H_1,\ldots ,H_k$ satisfying $\rk \big( \bigcap_{i\in I} H_i\pi \big) \geq 2$, for every $I\in \mathcal{I}$: since $\bigcap_{i\in I} H_i$ is not finitely generated, this follows immediately from \Cref{prop: mint fg iff} applied to the subgroups $H_i$ with $i\in I$. Alternatively, since $(\bigcap_{i\in I} H_i)\pi$ is not finitely generated and it is contained in $\bigcap_{i\in I} H_i\pi$, this last subgroup must be free nonabelian. 
\end{rem}

\section{\boldmath Realizing $k$-configurations}\label{sec: realizing}

Our goal in this section is to show that, for every $k\geq 1$, any $k$-configuration $\conf$ is realizable in $\FTA$, for big enough $m$. Note that, for this purpose, we can always assume $n=2$. Proposition~\ref{prop: almost0} already shows an interesting family of realizable $k$-configurations. It is straightforward to see that, conveniently adding trivial subgroups, it can be restated as follows.

\begin{lem} \label{lem: pretotal2}
Let $n\geq 2$. For every nonempty subset $I_0 \subseteq [k]$, the almost-$\zero$ $k$-configuration $\conf[I_0]$ is realizable in $\Free[n]\times \ZZ^{|I_0|-1}$ by subgroups $H_1,\ldots ,H_k$ further satisfying $\rk \big( \bigcap_{i\in I} H_i\pi \big) \neq 1$ for every nonempty $I\subseteq [k]$.
\end{lem}

\begin{proof}
Without loss of generality, we may assume that $I_0=[r]=\{1,\ldots ,r\}$, where $1\leq r=|I_0|\leq k$. By \Cref{prop: almost0}, the $r$-configuration $\conf[{I_0}]$ is realizable in $\Free[2]\times \ZZ^{r-1}$ by subgroups $H_1,\ldots ,H_r\leqslant \Free[2]\times \ZZ^{r-1}$ satisfying that $\rk \big( \bigcap_{j=1}^r H_j\pi \big) \geq 2$ (as it is mandatory according to~\Cref{prop: must}). Let $H_{r+1}=\cdots =H_{k}= \Trivial$. For any nonempty $I\subseteq [k]$ let us look at $H_I$ by distinguishing three cases:

\emph{Case 1: $I=[r]$.} By construction, $H_I$ is not finitely generated;

\emph{Case 2: $ I\subsetneq [r]$.} By construction, $H_I$ is finitely generated.

\emph{Case 3: $I\not\subseteq [r]$.} In this case, $H_I=H_{I\cap [r]}\cap \{1\}=\{1\}$ is obviously finitely generated.

\noindent Hence, we have realized the $k$-configuration $\conf[I_0]$ in $\Free[2]\times \ZZ^{r-1}\leqslant \Free[n]\times \ZZ^{|I_0|-1}$ by subgroups $H_1,\ldots ,H_k$ such that, for every nonempty $I\subseteq [k]$, $\bigcap_{i\in I} H_i\pi$ is either trivial or has rank $\geq 2$; therefore, $\rk \big( \bigcap_{i\in I} H_i\pi \big)\neq 1$, as required.
\end{proof}

Note that, in order to realize $\conf\join \conf'$ in a group $G$, we need to find $k$ subgroups $H_1,\ldots ,H_k\leqslant G$ such that, for every nonempty $ I\subseteq [k]$, $H_I$ is finitely generated whenever $(I)\conf =\fg$ \emph{and} $(I)\conf'=\fg$, and it is not finitely generated whenever $(I)\conf=\nfg$ \emph{or} $(I)\conf'=\nfg$.
Equivalence \eqref{eq: technical} matches the intended purpose. Note, however, that in order to use it (\Cref{thm: technical}), we need to include the technical hypothesis $\min(r',r'')\neq 1$, since, as shown in \Cref{ex: rk1}, equivalence \eqref{eq: technical} may fail if $\min(r',r'') = 1$.
Taking this into account, the result below provides a way to realize in the family of FTFA groups the join of two already realizable configurations (at some abelian cost).

\begin{prop}\label{prop: sum general}
Let $\conf'$ be a $k$-configuration realizable in $\Free[n']\times \ZZ^{m'}$ by $H'_1,\ldots ,H'_k$, and $\conf''$ be a $k$-configuration realizable in $\Free[n'']\times \ZZ^{m''}\!$ by $H''_1,\ldots ,H''_k$; and, for every nonempty $I\subseteq [k]$, let $r'_I=\rk \big( \bigcap_{i\in I} H'_i\pi \big)$ and $r''_I=\rk \big( \bigcap_{i\in I} H''_i\pi \big)$. If $\min\{r'_I,r''_I\} \neq 1$ for every $I\subseteq [k]$ with $|I|\geq 2$, then $\conf'\join \conf''$ is realizable in $\Free[n'+ n''] \times \ZZ^{m'+ m''}$.
\end{prop}

\begin{proof}

Under the assumptions of the statement, consider the subgroups 
 \[
H_1=\langle H'_1, H''_1\rangle, \ldots ,H_k
=
\langle H'_k, H''_k\rangle 
\leqslant 
(\Free[n']\times \ZZ^{m'})\cast (\Free[n'']\times \ZZ^{m''})
\isom 
\Free[n'+n'']\times \ZZ^{m'+m''}.
 \]

For every singleton $I=\{i\}\subseteq [k]$, it is clear that $H_i$ is finitely generated if and only if both $H'_i$ and $H''_i$ are finitely generated. And, for each $I\subseteq [k]$ with $|I|\geq 2$, we apply~\Cref{thm: technical} to obtain that 
\[
\bigcap\nolimits_{i\in I} H_i \text{ is f.g.} \ \Leftrightarrow \ \text{both } \bigcap\nolimits_{i\in I} H'_i \text{ \ and \ } \bigcap\nolimits_{i\in I} H''_i \text{ are f.g.}
 \]
This means that $H_1,\ldots ,H_k$ realize the $k$-configuration $\conf'\join\conf''$ in the group $\Free[n'+n'']\times \ZZ^{m'+m''}$. 
\end{proof}

\begin{cor}
Let $n\geq 2$, $k\geq 1$, and $i \in [k]$. If a $k$-configuration $\conf$ is realizable in $\Fn \times \Zm$, then $\conf \join \conf[\{i\}]$ is also realizable in $\Fn \times \Zm$; the converse is not true, in general. 
\end{cor}

\begin{proof}
Take two free nonabelian subgroups $F', F''\leqslant \Free[n]$ in free factor position, $\gen{F', F''}=F'*F''$. Let $H'_1,\ldots,H'_k \leqslant G'=F'\times \ZZ^m$ be a realization of $\conf$,
and let $H''_i\leqslant G''=F''\times \ZZ^0$ be non-f.g. and $H''_j=1\leqslant G''$ for $j\neq i$, (of course, realizing $\conf[\{i\}]$). Since $\min\{r'_I,r''_I\}=0\neq 1$ for every $I\subseteq [k]$, $|I|\geq 2$, we can apply \Cref{prop: sum general} to obtain a realization of $\conf \join \conf[\{i\}]$ in $G'\cast G''\leqslant \FTA$.  

The converse is not true since the 2-configuration $\conf[\{\{1\}, \{1,2\}\}]$ is realizable in $\Fn$, whereas $\conf[\{1,2\}]$
is not.
\end{proof}

Below we present a variation of \Cref{prop: sum general} which will be crucial in order to use this result inductively in the proof of \Cref{thm: all realizable}.

\begin{cor}\label{cor: sum}
Let $\conf'$ be realizable in $\Free[n']\times \ZZ^{m'}$ by subgroups $H'_1,\ldots ,H'_k$ satisfying $r'_I
\neq 1$, for all $I\subseteq [k]$ with $|I|\geq 2$, and let $\conf''$ be realizable in $\Free[n'']\times \ZZ^{m''}$ by subgroups $H''_1,\ldots ,H''_k$ satisfying $r''_I
\neq 1$, for all $I\subseteq [k]$ with $|I|\geq 2$. Then, $\conf'\join \conf''$ is realizable in $\Free[n'+n'']\times \ZZ^{m'+m''}$, by the subgroups $H_1=\gen{H'_1,H''_1},\ldots,H_k=\gen{H'_k,H''_k}$, satisfying (again) $r_I =\rk \big( \bigcap_{i\in I} H_i\pi \big) \neq 1$, for all $I\subseteq [k]$ with $|I|\geq 2$. 
\end{cor}

\begin{proof}
Note that we are under slightly stronger conditions that in \Cref{prop: sum general}; hence, the realizability of $\conf'\join \conf''$ follows immediately. Moreover, by \Cref{prop: = for free}, for each $I\subseteq [k]$ with $|I|\geq 2$, $\bigcap_{i\in I} H_i\pi =\big( \bigcap_{i\in I} H'_i\pi \big) * \big( \bigcap_{i\in I} H''_i\pi \big)$. Since, by hypothesis, $\rk\big( \bigcap_{i\in I} H'_i\pi \big) \neq 1$ and $\rk \big( \bigcap_{i\in I} H''_i\pi \big)\neq 1$ , we conclude that $\rk \big( \bigcap_{i\in I} H_i\pi \big)\neq 1$, as claimed. 
\end{proof}

Finally, iterating \Cref{cor: sum}, we obtain the main results from this section.

\begin{thm}\label{thm: all realizable}
Every finite configuration $\conf[\mathcal{I}]$ is realizable in $\Fn\times \ZZ^m$, for $m \geq  \sum_{I\in \mathcal{I}} (|I|-1)$ and~$n\geq 2$.
\end{thm}

\begin{proof}
Let $\conf$ be a $k$-configuration. If $\conf$ is the $\zero$ $k$-configuration ($\conf =\conf[\varnothing]$) then it can be realized by the trivial subgroups $H_1=\cdots =H_k=\{1\}\leqslant \Free[n]\times \mathbb{Z}^0$. If it is an almost-$\zero$ $k$-configuration $\conf =\conf[I_0]$ then, by Lemma~\ref{lem: pretotal2}, it can be realized by subgroups $H_1,\ldots ,H_k \leqslant \Free[n]\times \ZZ^{|I_0|-1}$, further satisfying $\rk \big( \bigcap_{i\in I} H_i\pi \big) \neq 1$, for each nonempty $I\subseteq [k]$.

In any other case, let $(1)\conf\preim=\{I_1, \ldots, I_r\}$ be the support of $\conf$, and we can decompose $\conf$ as the join of the corresponding almost-$\zero$ $k$-configurations $\conf=\conf[{I_1}]\join \cdots \join \conf[{I_r}]$. Now, realize each $\conf[{I_i}]$ in $\Free[n]\times \ZZ^{|I_i|-1}$ as above, and repeatedly apply \Cref{cor: sum}, to get a realization of $\conf$ in $\Free[rn]\times \ZZ^{|I_1|+\cdots +|I_r|-r}\leqslant \Free[n]\times \ZZ^{m}$.
\end{proof}

\begin{rem}
The proof of \Cref{thm: all realizable} actually shows a bit more than the statement: any $k$-configuration $\conf$ is realizable in $\Fn\times \ZZ^m$, for every $n\geq 2$ and $m\geq \sum_{(I)\conf=\nfg} (|I|-1)$, \emph{by subgroups $H_1,\ldots ,H_k$ further satisfying that, for every $I\subseteq [k]$ with $|I|\geq 2$, $\rk \big( \bigcap_{\, i\in I} H_i\pi \big) \neq 1$}. 
\end{rem}

\begin{exm}
Let $\mathcal{I}=\{\{1\}, \{2,3\}, \{1,3,4\}, \{2,3,4\}\}$ and consider the 4-configuration  $\conf=\conf[\mathcal{I}]$.  
Let us follow the previous inductive argument to realize $\conf$ in $\Free[2]\times \ZZ^m$ for big enough $m$. We have to find a family of four subgroups $\mathcal{H}=\{ H_1,H_2,H_3,H_4\}$ of $\Free[2]\times \ZZ^m$ such that $\conf^{\mathcal{H}}=\conf$.  

Decomposing $\conf$ as the join of the corresponding almost-$\zero$ $4$-configurations, we have 
 \[
\conf=\conf[{\{1\}}]\join \conf[{\{2,3\}}]\join \conf[{\{1,3,4\}}]\join \conf[{\{2,3,4\}}].
 \]
Let $\Free[2]=\langle x,y \mid -\rangle$ and consider the bi-infinite list of freely independent words $u_j=y^{-j}xy^j\in \Free[2]$, for $j\in \ZZ$. For the abelian part, take $m=0+1+2+2=5$ and let $\{\mathbf{e_1}, \mathbf{e_2}, \mathbf{e_3}, \mathbf{e_4}, \mathbf{e_5}\}$ be the canonical basis for $\ZZ^5$. Using  Lemma~\ref{lem: pretotal2}, we can realize the almost-$\zero$ 4-configurations individually in the following way:
\begin{itemize}
\item $\conf[\{1\}]$ as $H'_1=\langle \ldots, u_{-2}, u_{-1}\rangle$, $H'_2=\{1\}$, $H'_3=\{1\}$, $H'_4=\{1\}$, all viewed as subgroups of $G'=\langle ...,u_{-2},u_{-1}; -\rangle \leqslant \Free[2]\times \ZZ^5$; 
\item $\conf[\{2,3\}]$ as $H''_1=\{1\}$, $H''_2=\langle u_0, u_1 \rangle$, $H''_3=\langle u_0, u_1\t^\mathbf{e_1} \rangle$, $H''_4=\{1\}$, all viewed as subgroups of $G''=\langle u_0, u_1; \t^\mathbf{e_1}\rangle \leqslant \Free[2]\times \ZZ^5$;
\item $\conf[\{1,3,4\}]$ as $H'''_1=\langle u_2, u_3; \t^\mathbf{e_3}\rangle$, $H'''_2=\{1\}$, $H'''_3=\langle u_2, u_3; \t^\mathbf{e_2} \rangle$, $H'''_4=\langle u_2, u_3\t^\mathbf{e_2}; \t^\mathbf{e_3-e_2}\rangle$, all viewed as subgroups of $G'''=\langle u_2, u_3; \t^\mathbf{e_2}, \t^\mathbf{e_3}\rangle \leqslant \Free[2]\times \ZZ^5$; 
\item $\conf[\{2,3,4\}]$ as $H''''_1=\{1\}$, $H''''_2=\langle u_4, u_5; \t^\mathbf{e_5}\rangle$, $H''''_3=\langle u_4, u_5; \t^\mathbf{e_4} \rangle$, $H''''_4=\langle u_4, u_5\t^\mathbf{e_4}; \t^\mathbf{e_5-e_4}\rangle$, all viewed as subgroups of $G''''=\langle u_4, u_5; \t^\mathbf{e_4}, \t^\mathbf{e_5}\rangle \leqslant \Free[2]\times \ZZ^5$.
\end{itemize}

Note that, as stated in Lemma~\ref{lem: pretotal2}, each of these realizations is by subgroups whose intersections of projections to the free part are never cyclic; more precisely, for each nonempty $I\subseteq [k]$, we have $\rk \big( \bigcap_{i\in I} H'_i\pi \big) \neq 1$, $\rk \big( \bigcap_{i\in I} H''_i\pi \big) \neq 1$, $\rk \big( \bigcap_{i\in I} H'''_i\pi \big) \neq 1$, and $\rk \big( \bigcap_{i\in I} H''''_i\pi \big) \neq 1$. Moreover, we already took care to choose these realizations in strongly complementary subgroups of $\Free[2]\times \ZZ^5$, namely $G'$, $G''$, $G'''$, and $G''''$, respectively. Therefore, we can repeatedly apply \Cref{cor: sum} to get the following realization of $\conf=\conf[{\{1\}}]\join \conf[{\{2,3\}}]\join \conf[{\{1,3,4\}}]\join \conf[{\{2,3,4\}}]$ in $G'\cast G''\cast G'''\cast G''''\leqslant \Free[2]\times \ZZ^5$:
\begin{align*}
H_1 & \,=\, \langle \ldots, u_{-2}, u_{-1},u_{2}, u_3; \t^\mathbf{e_3}\rangle, \\ H_2 & \,=\, \langle u_0, u_1, u_4, u_5; \t^\mathbf{e_5}\rangle, \\ H_3 & \,=\, \langle u_0, u_1\t^\mathbf{e_1}, u_2, u_3, u_4, u_5; \t^\mathbf{e_2}, \t^\mathbf{e_4}\rangle, \\ H_4 & \,=\, \langle u_2, u_3\t^\mathbf{e_2}, u_4, u_5\t^\mathbf{e_4}; \t^\mathbf{e_3-e_2}, \t^\mathbf{e_5-e_4}\rangle. \tag*{\qed}
 \end{align*}
\end{exm}

This result tells us that every configuration is realizable collectively in the family of groups $\FTA$, i.e., in $\Free[2]\times \mathbb{Z}^m$ for $m$ big enough. Hence, in a group $G$ containing all of $\Free[2]\times \mathbb{Z}^m$ as subgroups it will be possible to realize \emph{any} configuration. 

\begin{thm}\label{int. sat.}
There exist finitely presented intersection-saturated groups.
\end{thm}

\begin{proof}
Recall that we denote by $\mathbb{Z}^{\infty}$ the direct sum of countably many copies of $\mathbb{Z}$, namely $\mathbb{Z}^{\infty}=\oplus_{n=1}^{\infty} \mathbb{Z}$. From \Cref{thm: all realizable}, it is enough to construct a finitely presented group $G$ containing $\Free[2]\times \mathbb{Z}^{\infty}$. 

Consider Thompson's group $F$: it is well-known that $F$ is finitely presented (in fact, it admits a presentation with just two generators and two relations, see \cite[Thm. 3.4]{cannon_introductory_1996}), and that it contains $\mathbb{Z}^{\infty}$, see \cite[Thm. 4.8]{cannon_introductory_1996}. Therefore, $G=\Free[2]\times F$ is a finitely presented intersection-saturated group. (Note that Thomson's group contains no nonabelian free subgroup so, in order to get intersection-saturation via our arguments, it is necessary to take the direct product with $\Free[2]$.) 

Alternatively, we can consider $\bigoplus_{n=-\infty}^{\infty} \ZZ =\langle \ldots ,x_{-1}, x_0, x_1, \ldots \mid [x_i, x_j]=\trivial,\,\, \forall i,j\in \mathbb{Z}\rangle$, take the automorphism $\varphi$ given by translation of coordinates, $x_i\mapsto x_{i+1}$, and take the semidirect product $G'=\mathbb{Z}^{\infty} \rtimes_{\varphi} \mathbb{Z}$. This group is finitely generated (in fact, by just $x_0$ and the stable letter $t$) and (not finitely but) recursively presented; so, by Higman's embedding theorem, it embeds in a finitely presented group $G'\into G$. Clearly, $\Free[2]\times G$ is finitely presented as well, and intersection-saturated.
\end{proof}

\section{Characterization for the free group case}\label{sec: free}

In this section we characterize the configurations which are realizable in a free group $\Fn$, $n\geq 2$. Roughly speaking, they are precisely those satisfying the Howson property; see Theorem~\ref{thm: char Fn}. That is, the Howson property is the only true obstacle for intersection realizability in a free nonabelian group. We first introduce a couple of notions convenient to state and prove our last result.

\begin{defn}
A $k$-configuration $\conf$ is said to be \emph{Howson} if, for every nonempty $I,J\subseteq [k]$, we have $(I\cup J)\conf=\fg$ whenever $(I)\conf=(J)\conf=\fg$.
\end{defn}

\begin{figure}[H]
\centering
\begin{tikzpicture}[>=latex,state/.style={circle,draw,minimum size=4mm},decoration={snake, segment length=2mm, amplitude=0.5mm,post length=1.5mm}]
    \node (1) {};
    \node[state] (10) [left = 1 of 1] {};
    \node[state] (01) [right = 1 of 1] {};    
    \node[state,fill=black] (11) [below = 0.5 of 1] {};
    \path[->] (01) edge[decorate] (11);
    \path[->] (10) edge[decorate] (11);
\end{tikzpicture}
\caption{A configuration is \emph{Howson} if and only if it \emph{does not} contain this pattern (where the snaked arrows denote directed paths, and the black vertex is where both paths first meet)} \label{fig:my_label}
\end{figure}

\begin{defn}
Let $\conf$ be a $k$-configuration with $k\geq 2$, and let $i\in [k]$. The \emph{restriction of $\conf$ to $\widehat{i} = [k]\setminus \{i\}$} is the $k-1$ configuration obtained from $\conf$ by removing index $i$ from the ambient, and restricting $\conf$ to the nonempty subsets of $[k]$ not containing $i$, \ie
 \[
\begin{array}{rcl}
\conf[|\hspace{1pt}\widehat{i}] \colon \mathcal{P}([k]\setminus \{i\})\setminus \{\varnothing\} & \to & \{\fg,\, \nfg\} \\ I\,\, & \mapsto & (I)\conf\,. \end{array}
 \]
Obviously, if $\conf$ is realizable in a group $G$ then so is $\conf[|\hspace{1pt}\widehat{i}]$\,.
\end{defn}

\begin{defn}
Let $\conf$ be a $k$-configuration, and let $i\in [k]$. The index $i$ is said to be \emph{$\zero$-chromatic} (in $\conf$) if $(I)\conf=0$ for every $I\subseteq [k]$ containing $i$; in other words, if $\conf =\conf[\mid\,\widehat{i}] \,\qplus_0 \zero$. Similarly, the index $i$ is said to be \mbox{\emph{$\one$-chromatic}} (in $\conf$) if $\conf= \conf[\mid\,\widehat{i}] \qplus_1 \one$.
\end{defn}

Below we use a typical property of free groups to realize intersection configurations `as deep as desired' into a FTFA group: any configuration realizable in $\FTA$, can always be realized using only a subgroup of $\Free[n]$ of some desired rank (at least two), and admitting a free supplement of arbitrary rank. 

\begin{lem} \label{lem: conf deep}
If a configuration $\conf$ is realizable in a FTFA group $\FTA$ with $n\geq 2$ then, for every $2\leq r\leq \infty$ and every $0\leq s\leq \infty$, there exist subgroups $F,K\leqslant \Free[n]$ such that $\rk(F)=r$, $\rk(K)=s$, and $\gen{F, K}=F*K\leqslant \Free[n]$, such that $\conf$ is also realizable in $F\times \Zm \leqslant \Fn \times \Zm$.
\end{lem}

\begin{proof}
This is an immediate consequence of the fact that $\Free[\infty]$ embeds in $\Fn$, for all $n \geq 2$.
\end{proof}

Observe that, in the particular case $m'= m''= 0$ (corresponding to free groups), the claim in \Cref{prop: sum general} follows immediately from \Cref{prop: = for free}, \emph{without any hypothesis} on the ranks of the subgroups.

\begin{prop}\label{prop: jjj}
Let $\conf'$ be a $k$-configuration realizable in $\Free[n']$ and $\conf''$ be a $k$-configuration realizable in $\Free[n'']$. Then, $\conf'\join \conf''$ is realizable in $\Free[n'+n'']$. \qed
\end{prop}

\begin{cor}\label{sum realizability}
If a $k$-configuration $\conf$ is realizable in $\Free[n]$ with $n\geq 2$, then the $(k+1)$-configurations $\conf \qplus_0 \zero$, $\conf \qplus_1 \one$, $\conf \qplus_{0} \conf$, and $\conf \qplus_{1} \conf$ are also realizable in $\Free[n]$.
\end{cor}

\begin{proof}
Apply~\Cref{lem: conf deep} with $m=0$, $r=\infty$, and $s=2$. Let $\{u, v\}$ be a free basis for the subgroup~$K$, and let $H_1,\ldots , H_k$ be a family of subgroups realizing $\conf$ in $F\leqslant \Free[n]$. Now, in order to realize:
\begin{itemize}
\item $\conf \qplus_0 \zero$, it is enough to take the subgroups $\widetilde{H}_1=H_1,\ldots , \widetilde{H}_k=H_k$, and $\widetilde{H}_{k+1}= \Trivial$;
\item $\conf \qplus_1 \one$, it is enough to take $\widetilde{H}_1=H_1*\gen{u,v},\ldots , \widetilde{H}_k=H_k*\gen{u,v}$ and $\widetilde{H}_{k+1}=\normalcl{v}_K$:
 for every $i \neq k+1$, $\widetilde{H}_{k+1} \cap \widetilde{H}_i = \widetilde{H}_{k+1}$ which is non-f.g., and
by \Cref{prop: jjj}, $\widetilde{H}_1,\ldots ,\widetilde{H}_k$ realize $\conf \join \zero=\conf$;
\item $\conf \qplus_0 \conf$, it is enough to take $\widetilde{H}_1=H_1,\ldots , \widetilde{H}_k=H_k$, and $\widetilde{H}_{k+1}=\Free[n]$.
\item $\conf \qplus_1 \conf$, it is enough to take $\widetilde{H}_1=H_1,\ldots , \widetilde{H}_k=H_k$, and $\widetilde{H}_{k+1}=F$.
\end{itemize}
This completes the proof.
\end{proof}

\begin{thm}\label{thm: char Fn}
A  finite configuration is realizable in $\Fn$, with $n\geq 2$, if and only if it is Howson. 
\end{thm}

\begin{proof}
For all the proof, we can assume $n=2$. Clearly, $\conf$ being Howson is a necessary condition. To show the converse, we will do induction on the cardinal of the support of $\conf$, say $s$ (regardless of its size~$k$). If $s=0$ then $\conf$ is the $\zero$ $k$-configuration, which is clearly realizable in $\Free[2]$ (in fact, in any group).

Suppose that every Howson configuration with support of size strictly less than $s\geq 1$ is realizable in $\Free[2]$. Let $\conf$ be a Howson $k$-configuration with support of cardinal $s$, and let us realize it in~$\Free[2]$. 

We define the \emph{cone of $\conf$ with vertex $I\subseteq [k]$}, denoted by $\cone{\conf}{I}$, as the $k$-configuration obtained after removing from the support of $\conf$ all the sets of indices not contained in $I$, if any; i.e.,
\[
\begin{array}{rcl} \cone{\conf}{I}\colon {\mathcal P}([k])\setminus \{\varnothing\} & \to & \{\fg,\, \nfg\} \\[2pt] J\,\, & \mapsto & \left\{ \begin{array}{ll} \fg & \text{if } J\not\subseteq I, \\ (J)\conf & \text{if } J\subseteq I. \end{array}\right. \end{array}
 \]

Now let $I_1, \ldots ,I_p\subseteq [k]$ be the maximal elements (with respect to inclusion) in the support of~$\conf$. It is clear that $\conf=\cone{\conf}{I_1}\join \cdots \join \cone{\conf}{I_p}$. If $p\geq 2$, by the induction hypothesis we can realize each of $\cone{\conf}{I_1}, \ldots ,\cone{\conf}{I_p}$ in $\Free[2]$, and by \Cref{prop: jjj}, we can realize their join $\conf$, in $\Free[2]$ as well. 

Hence, we are reduced to the case $p=1$, i.e., $\conf$ is a Howson $k$-configuration for which there is a nonempty set of indices $I_1 \subseteq [k]$ with $(I_1)\conf =1$, and $(J)\conf=\fg$ for every $J\not\subseteq I_1$. If $I_1\neq [k]$ then any index $j\in [k]\setminus I_1$ is $\zero$-chromatic and, by Corollary~\ref{sum realizability}, in order to realize $\conf =\conf[\mid\,\widehat{j}] \qplus_0 \zero$ we are reduced to realize the restriction $\conf[\mid\,\widehat{j}]$\,; repeating this operation for all such indices, we are reduced to the case $I_1=[k]$. That is, $\conf$ is a Howson $k$-configuration such that $([k])\conf=\nfg$.

If every nonempty $I\subseteq [k]$ satisfies $(I)\conf=\nfg$ then $\conf =\one$ and so, it is realizable in $\Free[2]$. Otherwise, take $\varnothing \neq I_2\subseteq [k]$ with $(I_2)\conf=\fg$ and with maximal possible cardinal. Since $I_2 \neq [k]$, there exist indices $j\not\in I_2$. And any such index $j$ is $\one$-chromatic: in fact, any subset $J\subseteq [k]$ containing $j$ satisfies $|I_2\cup J|>|I_2|$ so $(I_2\cup J)\conf=\nfg$ and, since $\conf$ is Howson and $(I_2)\conf = 0$, then $(J)\conf=\nfg$. Hence, by induction hypothesis, \smash{$\conf[\mid\, \widehat{j}]$} is realizable in $\Free[2]$ and, by Corollary~\ref{sum realizability}, $\conf=\conf[\mid\,\widehat{j}] \qplus_1 \one$ is also realizable in $\Free[2]$. This concludes the proof.
\end{proof}

As it is clear from the previous argument, all the steps in the proof of \Cref{thm: char Fn} are constructive. The result below follows.

\begin{cor}
There is an algorithm which, on input a Howson $k$-configuration $\conf$, provides explicit generators for subgroups $H_1,\ldots ,H_k\leqslant \Free[2]$ realizing $\conf$. \qed
\end{cor}

\section{Open questions}

We finish by asking three related natural questions.

\begin{que}
Is the obstruction in Proposition~\ref{prop: main obstruction} the only one for a configuration to be realizable in $\FTA$ for a particular abelian dimension $m$?
\end{que}

\begin{que}
Which $k$-configurations $\conf$ are realizable in $\FTA$ for any fixed $m$? Is it possible to give an explicit characterization in terms of $m$? Or at least an algorithm which, on input $\conf$ and $m$, decides whether $\conf$ is realizable in $\FTA$ (and, in the affirmative case, computes such a realization)? 
\end{que}

\begin{que}
Is there a finitely presented intersection-saturated group $G$ which does not contain~${\Free[2] \times \ZZ^{m}}$, for some $m\in \NN$?
\end{que}

\section*{Acknowledgements}

We are grateful to the anonymous referee for the detailed report on the initial version of this paper, which pointed out an overly coarse approach to the proof of the delicate \Cref{thm: technical}. This feedback enabled us to refine it to its current final form.

The three authors acknowledge support from the Spanish Agencia Estatal de Investigaci\'on through grant MTM2017-82740-P (AEI/ FEDER, UE). The first named author was also partially supported by MINECO grant PID2019-107444GA-I00 and the Basque Government grant~IT974-16,
and thanks for the support provided from the Universitat Politècnica de Catalunya through a María Zambrano grant. The second named author would like to express gratitude for the hospitality received from the Universidad del País Vasco (UPV/EHU) and for the support provided through a Margarita Salas grant from the Universitat Politècnica de Catalunya.

\renewcommand*{\bibfont}{\small}
\printbibliography

\end{document}